\documentclass[10pt,a4paper]{amsart}

\numberwithin{equation}{section}
\allowdisplaybreaks 

\usepackage{mathtools,amssymb,eucal,mathrsfs,setspace,graphicx} 
\usepackage[noadjust]{cite} 
\usepackage[margin=22mm,foot=10mm]{geometry}
\usepackage{lipsum}
\usepackage{tikz-cd} 

\usepackage{array}
\newcolumntype{C}{>{$}c<{$}} 

\usepackage[largesc,theoremfont]{newtxtext}      
\usepackage[libertine,cmbraces,varbb]{newtxmath} 
\begingroup
  \makeatletter
  \@for\theoremstyle:=definition,remark,plain\do{%
    \expandafter\g@addto@macro\csname th@\theoremstyle\endcsname{%
      \addtolength\thm@preskip{.5\baselineskip plus .2\baselineskip minus .2\baselineskip}
      \addtolength\thm@postskip{.5\baselineskip plus .2\baselineskip minus .2\baselineskip}
    }%
  }
\endgroup

\usepackage{enumitem}
\setitemize{leftmargin=*}   
\setenumerate{leftmargin=*, 
  label=\textup{(\alph*)},  
	ref=\textup{\alph*}}      
\usepackage{tikz}
\usetikzlibrary{arrows}
\usetikzlibrary{arrows.meta}
\usetikzlibrary{shapes}
\usepackage{bbm}
\usepackage[colorlinks=true,citecolor=red,linkcolor=blue]{hyperref} 
\usepackage{marginnote}
\usepackage[capitalise,noabbrev]{cleveref} 

\newcommand{\pd}{\partial}     

\renewcommand{\ge}{\geqslant} 
\renewcommand{\le}{\leqslant} 

\renewcommand{\cong}{\simeq} 

\DeclareMathOperator{\id}{id}

\newcommand{\cc}{\mathsf{c}}   
\newcommand{\ee}{\mathsf{e}}   
\newcommand{\kk}{\mathsf{k}}   
\newcommand{\uu}{\mathsf{u}}   
\providecommand{\vv}{\mathsf{v}}\renewcommand{\vv}{\mathsf{v}}   

\newcommand{\fracuv}{\frac{\uu}{\vv}} 

\newcommand{\wun}{\vvmathbb{1}}  

\DeclarePairedDelimiter{\brac}{\lparen}{\rparen}   
\DeclarePairedDelimiter{\sqbrac}{\lbrack}{\rbrack} 
\DeclarePairedDelimiter{\set}{\lbrace}{\rbrace}
\newcommand{\st}{\mspace{5mu} {:} \mspace{5mu}}    

\newcommand{\no}[1]{\mathopen{:} #1 \mathclose{:}} 

\DeclarePairedDelimiterX{\comm}[2]{\lbrack}{\rbrack}{#1 , #2}  
\DeclarePairedDelimiterX{\acomm}[2]{\lbrace}{\rbrace}{#1 , #2} 
\DeclarePairedDelimiterX{\inner}[2]{\langle}{\rangle}{#1 , #2} 
\DeclarePairedDelimiterX{\super}[2]{\lparen}{\rparen}{#1 \delimsize\vert \mathopen{} #2} 

\newcommand{\fld}[1]{\mathbb{#1}}    
\newcommand{\alg}[1]{\mathfrak{#1}}  
\newcommand{\VOA}[1]{\mathsf{#1}}    
\newcommand{\categ}[1]{\mathscr{#1}} 

\newcommand{\ZZ}{\fld{Z}}

\newcommand{\RR}{\fld{R}}
\newcommand{\CC}{\fld{C}}

\newcommand{\affine}[1]{\widehat{#1}}

\newcommand{\SLA}[2]{\alg{#1}_{#2}}                      

\newcommand{\sltwo}{\SLA{sl}{2}}
\newcommand{\slthree}{\SLA{sl}{3}}
\newcommand{\slnpone}{\SLA{sl}{n+1}}

\newcommand{\fsub}{f_\textup{sub}}                                    
\newcommand{\freg}{f_\textup{reg}}                                    

\newcommand{\bpsymb}{\VOA{BP}}
\newcommand{\Wsymb}{\VOA{W}}
\newcommand{\Wthree}{\Wsymb_3}

\newcommand{\uaffvoa}[2]{\VOA{V}^{#1}\brac{#2}}                   
\newcommand{\saffvoa}[2]{\VOA{L}_{#1}\brac{#2}}                   

\newcommand{\uslvoa}[1]{\uaffvoa{#1}{\slthree}}                   
\newcommand{\ubpvoa}[1]{\bpsymb^{#1}}                             
\newcommand{\sbpvoa}[1]{\bpsymb_{#1}}                             

\newcommand{\uslvoak}{\uslvoa{\kk}}                               
\newcommand{\ubpvoak}{\ubpvoa{\kk}}                               
\newcommand{\sbpvoak}{\sbpvoa{\kk}}                               

\newcommand{\uWvoa}[1]{\Wthree^{#1}}                       		  
\newcommand{\uWvoak}{\uWvoa{\kk}}                                 

\newcommand{\halflattice}{\Pi}                          		  

\newcommand{\uReg}[2]{\Wsymb{}^{#1}_{#2}}                       	  
\newcommand{\sReg}[2]{\Wsymb_{#2, \ #1}}                   	      
\newcommand{\Regminmod}[2]{\Wsymb_{n+1} \brac{#1,#2}}             
\newcommand{\uRegkk}{\uReg{\kk}{n+1}}                             
\newcommand{\sRegkk}{\sReg{\kk}{n+1}}                             
\newcommand{\Regminmoduv}{\Regminmod{\uu}{\vv}}             	  

\newcommand{\uSub}[2]{\overline{\Wsymb}{}^{#1}_{#2}}                 
\newcommand{\uSubkk}{\uSub{\kk}{n+1}}                             
\newcommand{\sSub}[2]{\overline{\Wsymb}_{#2, \ #1}}                
\newcommand{\sSubkk}{\sSub{\kk}{n+1}}                             
\newcommand{\Subminmod}[2]{\overline{\Wsymb}_{n+1} \brac{#1,#2}}   
\newcommand{\Subminmoduv}{\overline{\Wsymb}_{n+1} \brac{\uu,\vv}}  

\newcommand{\Subcc}[1]{\overline{\cc}^{n+1}_{#1}}                 
\newcommand{\Regcc}[1]{\cc^{n+1}_{#1}}                            
\newcommand{\lcc}[1]{\cc^{\halflattice}_{#1}}                     
\newcommand{\Subcck}{\Subcc{\kk}}                                 
\newcommand{\Regcck}{\Regcc{\kk}}                                 
\newcommand{\lcck}{\lcc{\kk}}                                     
\newcommand{\lccuv}{\lcc{\uu,\vv}}                                

\newcommand{\picorr}[2]{\pi^{#1,#2}}                             

\newcommand{\uQHR}[3]{\Wsymb^{#1}\brac{#2,#3}}                    
\newcommand{\sQHR}[3]{\Wsymb_{#1}\brac{#2,#3}}                    

\newcommand{\FS}[1]{\mathcal{W}^{(2)}_{#1}}                       

\newcommand{\zhu}[1]{\mathsf{Zhu}\sqbrac*{#1}}                	  

\newcommand{\sfsymb}{\sigma}                             		  

\newcommand{\sfmod}[2]{\sfsymb^{#1}\bigl(#2\bigr)}       		  

\newcommand{\halflatticemod}[2]{\Pi_{#1}\brac[\big]{#2}}   
\newcommand{\tpmod}[2]{#1 (#2)}   

\DeclareMathOperator{\chmap}{ch}              

\newcommand{\Gr}[1]{\sqbrac[\big]{#1}}        

\newcommand{\charac}[3]{\chmap \Gr{#1} \left( #2 ; #3 \right)}
\newcommand{\qcharac}[2]{\chmap \Gr{#1} \left( #2 \right)}

\newcommand{\vo}{vertex operator}
\newcommand{\voa}{\vo\ algebra}
\newcommand{\voas}{\voa s}

\newcommand{\ope}{operator product expansion}
\newcommand{\opes}{\ope s}
\newcommand{\qhr}{quantum hamiltonian reduction} 

\theoremstyle{plain}
\newtheorem{theorem}{Theorem}[section]
\newtheorem{corollary}[theorem]{Corollary}
\newtheorem{lemma}[theorem]{Lemma}
\newtheorem{proposition}[theorem]{Proposition}

\newtheorem{mthm}{Main Result}

\newtheorem*{example}{Example}

\newtheorem{definition}[theorem]{Definition}

\Crefname{assumption}{Assumption}{Assumptions}

\DeclareRobustCommand{\SkipTocEntry}[5]{} 

\begin{document}

\title[Subregular W-algebras of type $A$]{Subregular W-algebras of type $A$}

\author[Z~Fehily]{Zachary Fehily}
\address[Zachary Fehily]{
School of Mathematics and Statistics \\
University of Melbourne \\
Parkville, Australia, 3010.
}
\email{zfehily@student.unimelb.edu.au}

\begin{abstract}
Subregular W-algebras are an interesting and increasingly important class of quantum hamiltonian reductions of affine vertex algebras. Here, we show that the $\mathfrak{sl}_{n+1}$ subregular W-algebra can be realised in terms of the  $\mathfrak{sl}_{n+1}$ regular W-algebra and the half lattice vertex algebra $\Pi$. This generalises the realisations found for $n=1$ and $2$ in \cite{AdaRea17,AdaRea20} and can be interpreted as an inverse quantum hamiltonian reduction in the sense of Adamovi\'c. We use this realisation to explore the representation theory of $\mathfrak{sl}_{n+1}$ subregular W-algebras. Much of the structure encountered for $\mathfrak{sl}_{2}$ and $\mathfrak{sl}_{3}$ is also present for $\mathfrak{sl}_{n+1}$. Particularly, the simple $\mathfrak{sl}_{n+1}$ subregular W-algebra at nondegenerate admissible levels can be realised purely in terms of the $\mathsf{W}_{n+1}$ minimal model vertex algebra and $\Pi$.
\end{abstract}

\maketitle

\onehalfspacing

\section{Introduction}

\subsection{Background}
Given a simple Lie superalgebra $\mathfrak{g}$, a complex number $\kk$ and a nilpotent element $f \in \mathfrak{g}$, there exists a cochain complex of vertex operator algebras whose zeroth cohomology has the structure of a vertex operator algebra $\uQHR{\kk}{\mathfrak{g}}{f}$ called a \emph{W-algebra} \cite{KacQua03, KacQua04}. This process is called \emph{quantum hamiltonian reduction} of the universal affine vertex algebra $\uaffvoa{\kk}{\mathfrak{g}}$. By choosing the nilpotent element $f$ carefully, one obtains vertex operator algebras that are particularly easy to work with. Historically, a popular choice of nilpotent element has been the \emph{regular} (or \emph{principal}) nilpotent element $\freg \in \mathfrak{g}$. The associated W-algebra $\uQHR{\kk}{\mathfrak{g}}{\freg}$ has been studied extensively and has been at the forefront of many developments in mathematics and physics. See \cite{Beem15a,Fre18,Ara19q,AGT10,Yam88} for example.

Recent work in physics and mathematics has indicated the importance of the W-algebra $\uQHR{\kk}{\mathfrak{g}}{\fsub}$ associated to \emph{subregular} nilpotent elements $\fsub \in \mathfrak{g}$. For example, subregular W-algebras appear in the Schur-index of $4D$ superconformal field theories known as Argyres--Douglas theories \cite{Arg95,Beem15,Cre17w}. Subregular nilpotent elements and their nilpotent orbits also play a crucial role in singularity theory: the ADE classification of simple surface singularities connects to the ADE classification of simply-laced Lie algebras through the geometry of the Slodowy slice corresponding to the subregular nilpotent orbit \cite{Slod80}. This connection is expressed beautifully in the associated variety of the subregular W-algebra \cite{AraAss15}. In light of these and many more motivations, much recent work has been done to improve our understanding of the structure and representation theory of subregular W-algebras \cite{AdaVer95,Feh21,CreLin17,Fasquel2020,GenStrong20,CreGen21,CreGenShig21}. 

In general, an understanding of W-algebras for nonregular nilpotents is highly desirable. One means of achieving this is to leverage the well-understood W-algebras to learn about others: It is strongly suspected that in addition to the usual quantum hamiltonian reduction, one can also perform a `partial quantum hamiltonian reduction' between two W-algebras as long as their corresponding nilpotent orbits are related by a certain partial ordering \cite{Collingwood17}. Such partial reductions are in the same spirit as the `secondary quantum hamiltonian reductions' introduced by Madsen and Ragoucy \cite{Ra96}. In fact, a partial reduction for finite W-algebras has been defined for certain cases \cite{MorQua15}. This is a good sign as finite W-algebras arise as the Zhu algebras of W-algebras \cite{DeSFin06}.

The philosophy behind this paper is that despite not having a complete understanding of how partial quantum hamiltonian reduction should be defined for W-algebras, it can be inverted. The first inverse quantum hamiltonian reduction was described by Semikhatov for the level-$\kk$ universal affine vertex algebra $\uaffvoa{\kk}{\sltwo}$ from a string-theoretic point of view \cite{SemInv94}. Later work by Adamovi\'c and collaborators showed how such inverse reductions can be used in the representation theory of W-algebras  \cite{AdaRea17,AdaRea20,FehilyMod,ACG21}. 

Owing to Figure \ref{fig:ordering}, the `simplest' class of W-algebras for which we expect there to be an inverse quantum hamiltonian reduction are the universal regular W-algebras $\uQHR{\kk}{\slnpone}{\freg}$ and the universal subregular W-algebras $\uQHR{\kk}{\slnpone}{\fsub}$. Concretely, what we expect is the existence of an embedding of the subregular W-algebra into the regular W-algebra tensored with some lattice vertex algebra. The lattice part represents the degrees of freedom lost during partial quantum hamiltonian reduction. Here, we prove that such an embedding exists for non-critical $\kk$ and describe it explicitly.

\begin{figure} 
	\begin{tikzpicture}[xscale=2.5,yscale=1.5]
		\draw[->] 
		(-2.9,1.9)  node[above] {$\uQHR{\kk}{\sltwo}{\fsub} = \uQHR{\kk}{\sltwo}{0} \cong \uaffvoa{\kk}{\sltwo}$} -- 
		(-2.9,1.1) node[below] {$\uQHR{\kk}{\sltwo}{\freg} \cong \VOA{Vir}^\kk$};
		
		\draw[->,blue] (-2.8,1.1) .. controls (-2.5,1.5) .. (-2.8,1.9) ;
		\node at (-2.39,1.5) {\cite{SemInv94,AdaRea17}};
		
		\draw[->,dashed] 
		(-0.5,0.68)  node[above] {$\uQHR{\kk}{\slthree}{f_\textup{min}} = \uQHR{\kk}{\slthree}{\fsub} \cong \ubpvoak$} -- 
		(-0.5,-0.12) node[below] {$\uQHR{\kk}{\slthree}{\freg} \cong \uWvoak$};
		
		\draw[->,blue] (-0.4,-0.12) .. controls (-0.1,0.28) .. (-0.4,0.68) ;
		\node at (-0.05,0.28) {\cite{AdaRea20}};
		
		\draw[->] 
		(-0.5,1.9)  node[above] {$\uQHR{\kk}{\slthree}{0} \cong \uaffvoa{\kk}{\slthree}$} -- 
		(-0.5,1.1);
		
		\draw[->,blue] (-0.4,1.1) .. controls (-0.1,1.5) .. (-0.4,1.9) ;
		\node at (-0.05,1.5) {\cite{ACG21}};
		
		\draw[->] (-0.6,1.9) .. controls (-2,0.89) .. (-0.6,-0.12);
		
		\draw[->] 
		(1.9,1.9)  node[above] {$\uQHR{\kk}{\slnpone}{0} \cong \uaffvoa{\kk}{\slnpone}$} --
		(1.9,1.1) node[below] {$\uQHR{\kk}{\slnpone}{f_\textup{min}}$};
		
		\draw[->,dashed] (1.95,0.68) -- (2.3,0);
		\draw[->,dashed] (1.85,0.68) -- (1.5,0);
		
		\node at (1.9,-0.2) {\vdots};
		\node at (2.3,-0.2) {\vdots};
		\node at (1.5,-0.2) {\vdots};
		
		\draw[->,dashed] (2.3,-0.6) -- (1.95,-1.28);
		\draw[->,dashed] (1.5,-0.6) -- (1.85,-1.28);
		
		\draw[->,dashed] 
		(1.9,-1.7)  node[above] {$\uQHR{\kk}{\slnpone}{\fsub}$} -- 
		(1.9,-2.7) node[below] {$\uQHR{\kk}{\slnpone}{\freg}$};
		
		\draw[->] (1.8,1.9) .. controls (0.85,0.41) .. (1.7,-1.25);
		\draw[->] (1.7,1.9) .. controls (0.4,-0.15) .. (1.8,-2.7);
		
		\draw[->,blue] (2,-2.7) .. controls (2.4,-2.2) .. (2,-1.7) ;
		
		\draw[->] (-3,-1) -- (-3,-1.4);
		\node at (-1.98,-1.2) {\small{\ = \ Quantum hamiltonian reduction}};
		\draw[->,dashed] (-3,-1.6) -- (-3,-2);
		\node at (-1.8,-1.8) {\small{\ = \ Partial quantum hamiltonian reduction}};
		\draw[->, blue] (-3,-2.6) -- (-3,-2.2);
		\node at (-1.78,-2.4) {\small{\ = \ Inverse quantum hamiltonian reduction}};
	\end{tikzpicture}
\caption{The partial ordering of W-algebras for $\sltwo$, $\slthree$ and $\slnpone$. The inverse quantum hamiltonian reduction we construct in this paper is represented by the unlabelled blue arrow.}
\label{fig:ordering}
\end{figure}

\noindent  Another notable feature of the regular-subregular example is that the simple regular W-algebra  $\sQHR{\kk}{\mathfrak{g}}{\freg}$ is rational \cite{AraRat15} for almost all admissible levels. At these levels, one goal is to construct relaxed $\sQHR{\kk}{\mathfrak{g}}{\fsub}$-modules out of the (finitely many)  $\sQHR{\kk}{\mathfrak{g}}{\freg}$-modules and eventually use these to construct `logarithmic minimal models' with admissible-level $\sQHR{\kk}{\mathfrak{g}}{\fsub}$ symmetry. This has been done for $\mathfrak{g} =$  $\sltwo$ \cite{CreMod13} and $\slthree$ \cite{FehilyMod}. In both cases, inverse quantum hamiltonian reduction played a crucial role.

The first step in this logarithmic conformal field theoretic direction is to construct and analyse such relaxed $\uQHR{\kk}{\slnpone}{\fsub}$-modules. Here, we consider these relaxed $\uQHR{\kk}{\slnpone}{\fsub}$-modules and use them to derive a number of results analogous to those found for $\sltwo$ and $\slthree$. 

Overall, the results of this paper further strengthen the suspicion that inverse quantum hamiltonian reductions are fundamental tools in the study of W-algebras, their representation theory and the associated conformal field theories.

\subsection{Results}
Our strategy for finding an inverse quantum hamiltonian reduction from $\uQHR{\kk}{\slnpone}{\freg}$ to $\uQHR{\kk}{\slnpone}{\fsub}$ is using free-field realisations and their corresponding screening operators. The screening operators of $\uQHR{\kk}{\slnpone}{\fsub}$ \cite{Gen17} are superficially similar to those of $\uQHR{\kk}{\slnpone}{\freg}$ \cite{FatLuk88} except for the presence of an additional factor related to a $\beta \gamma$ ghost system $\VOA{B}$. Dealing with this factor can be done using \emph{Friedan-Martinec-Shenker bosonisation} \cite{FMS86}. This is an embedding of $\VOA{B}$ into the half lattice vertex algebra $\Pi$ which has strong generating fields denoted by $c(z)$, $d(z)$ and $\ee^{m c}(z)$ with $m \in \ZZ$. It is convenient to work in terms of $a(z), b(z) \in \Pi$ defined in \eqref{eq:abbasis} rather than $c(z)$ and $d(z)$. 

\begin{mthm}[Theorem \ref{thm:invemb}]
Let $\kk$ be generic. There exists an embedding $\uQHR{\kk}{\slnpone}{\fsub} \hookrightarrow \uQHR{\kk}{\slnpone}{\freg} \otimes \Pi$ with
\begin{equation}
    \uQHR{\kk}{\slnpone}{\fsub} \cong \text{ \emph{ker}} \int  \ee^{a - \omega_1}(z)  \ dz,
\end{equation}
where $\ee^{a}(z)$ and $\ee^{-\omega_1}(z)$ are screening fields for $\Pi$ and $\uQHR{\kk}{\slnpone}{\freg}$ respectively.
\end{mthm}

\noindent That such an embedding exists for a general class of W-algebras is strong evidence that other inverse quantum hamiltonian reductions both exist and can be identified using free-field realisations \cite{Gen17}. 

The embedding of Theorem \ref{thm:invemb} is related to a free-field realisation of $\uQHR{\kk}{\slnpone}{\fsub}$, where explicit formulae for strong generating fields of $\uQHR{\kk}{\slnpone}{\fsub}$ have been obtained \cite{GenStrong20}. In order to leverage such formulae, it is necessary to use certain identifications of the vertex algebras involved that make the $\uQHR{\kk}{\slnpone}{\freg}$ factor manifest. 

Given these identifications, we can describe strong generators for $\uQHR{\kk}{\slnpone}{\fsub}$ in terms of fields in $\uQHR{\kk}{\slnpone}{\freg}$ and $\Pi$: Let $\{W_2(z), W_3(z), \dots, W_{n+1}(z) \}$ be the fields of the Miura transform of $\uQHR{\kk}{\slnpone}{\freg}$ defined in \eqref{eq:freefieldforregfields} and $T(z) = \frac{1}{\kk+n+1} W_2(z)$. Equip $\Pi$ with the conformal structure defined by the field $t(z) = \frac{1}{2} \no{c(z)d(z)} - \frac{n+1}{2} \pd a(z) + \frac{n-1}{2} \pd b(z)$. We will identify $\uQHR{\kk}{\slnpone}{\fsub}$ with the image of the embedding of Theorem \ref{thm:invemb}. Denote the the vacuum vectors of $\uQHR{\kk}{\slnpone}{\freg}$ and $\Pi$ by $\wun_R$ and $\wun_\Pi$ respectively.

\begin{mthm}[Proposition \ref{prop:stronggensubreg}, Theorem \ref{thm:regcontent}]
Let $\kk \neq -n-1$. The fields corresponding to the following $\uQHR{\kk}{\slnpone}{\freg} \otimes \Pi$ states are fields in $\uQHR{\kk}{\slnpone}{\fsub}$:
\begin{equation} 
\begin{gathered}
    L = T\otimes \wun_\Pi + \wun_R \otimes t, \quad J = \wun_R \otimes b, \quad G^+ = \wun_R \otimes \ee^{c}, \quad U_i = (-1)^{i+1} W_i \otimes \wun_\Pi + \sum_{j=0}^{i-1} W_j \otimes \picorr{i}{j}, \\
    G^- = W_{n+1} \otimes \ee^{-c} +   \sum_{j=0}^{n} W_j \otimes \left( \picorr{-}{j}_{(-1)} \ee^{-c} \right),
\end{gathered}
\end{equation}
where $\picorr{i}{j}(z) \in \Pi$ can be explicitly described in principle and
\begin{equation}
    \picorr{-}{j}(z) = \big( (\kk + n) t_{-1} + a_{-1} \big)^{n+1-j}\wun_\Pi(z).
\end{equation}
Moreover, the fields $\{G^\pm(z), J(z), L(z), U_3(z), \dots, U_n(z)\}$ strongly generate $\uQHR{\kk}{\slnpone}{\fsub}$.
\end{mthm}

\noindent A proof of the existence of the embedding $\uQHR{\kk}{\slnpone}{\fsub} \hookrightarrow \uQHR{\kk}{\slnpone}{\freg} \otimes \Pi$ at critical level, in addition to explicit formulae for strong generators, was obtained in \cite{GenStrong20}. Therefore the inverse quantum hamiltonian reduction from $\uQHR{\kk}{\slnpone}{\freg}$ to $\uQHR{\kk}{\slnpone}{\freg}$ exists for all $\kk$.

Analogous to the construction for $\saffvoa{\kk}{\sltwo}$ \cite{CreMod13} and $\ubpvoak$ \cite{AdaRea20}, the existence of an inverse quantum hamiltonian reduction embedding can be used to construct $\uQHR{\kk}{\slnpone}{\fsub}$-modules. In light of the explicit formulae above, we are able to choose appropriate $\uQHR{\kk}{\slnpone}{\freg}$- and $\Pi$-modules such that their tensor product is a positive-energy weight $\uQHR{\kk}{\slnpone}{\fsub}$-module by restriction. 

A particularly important class of $\uQHR{\kk}{\slnpone}{\fsub}$-modules consists of tensor products of irreducible $\uQHR{\kk}{\slnpone}{\freg}$-modules $M$ with relaxed $\Pi$-modules $\halflatticemod{-1}{\lambda}$ where $\lambda \in \CC$. The corresponding $\uQHR{\kk}{\slnpone}{\fsub}$-module is denoted by $\tpmod{M}{-1,\lambda} = M \otimes \halflatticemod{-1}{\lambda}$.

\begin{mthm}[Proposition \ref{prop:relaxedmodfromhw}]
Let $M$ be an irreducible $\uQHR{\kk}{\slnpone}{\freg}$-module. Then 
\begin{itemize}
\item $\tpmod{M}{-1,\lambda}$ is an indecomposable relaxed highest-weight $\uQHR{\kk}{\slnpone}{\fsub}$-module for all $\lambda \in \CC$.
\item $\tpmod{M}{-1,\lambda}$ is irreducible for almost all $\lambda \in \CC$.
\end{itemize}
\end{mthm}

\noindent The final main result of this paper is concerned with when the embedding of Theorem \ref{thm:invemb} descends to an embedding of simple quotients $\sQHR{\kk}{\slnpone}{\fsub} \hookrightarrow \sQHR{\kk}{\slnpone}{\freg} \otimes \Pi$. Despite having little-to-no general information about singular vectors in $\uQHR{\kk}{\slnpone}{\fsub}$, this question can be answered and manifests as a restriction on the level $\kk$.

\begin{mthm}[Theorem \ref{thm:descent}]
$\sQHR{\kk}{\slnpone}{\fsub}$ embeds into $\sQHR{\kk}{\slnpone}{\freg} \otimes \Pi$ if and only if $i(\kk+n) \notin \ZZ_{\ge 1}$ for all $i \in \{1, \dots, n\}$.
\end{mthm}

\noindent Combining the embedding of simple quotients with the construction of relaxed modules shows that, at the levels allowed by Theorem \ref{thm:descent}, $\sQHR{\kk}{\slnpone}{\fsub}$ admits infinitely many irreducible modules in addition to reducible-but-indecomposable modules. That is,  $\sQHR{\kk}{\slnpone}{\fsub}$ is nonrational at such levels. This includes nondegenerate admissible levels $\kk = \fracuv-n-1$ where $\uu, \vv \in \ZZ_{\ge n+1}$ and $\text{gcd}\{\uu,\vv \} = 1$. At these levels, $\sQHR{\kk}{\slnpone}{\freg}$ is rational and is known as the $\VOA{W}_{n+1}$ minimal model. 

As seen for $\saffvoa{\kk}{\sltwo} = \sQHR{\kk}{\sltwo}{\fsub}$ and $\sbpvoak \cong \sQHR{\kk}{\slthree}{\fsub}$, the fact that infinitely many relaxed modules can be realised in terms of $\VOA{W}_{n+1}$ minimal model modules allows for the construction of `nonrational minimal models'. Modular transformations and Grothendieck fusion rules for these nonrational minimal models are obtained using the standard module formalism. The data of the corresponding $\VOA{W}_{n+1}$ minimal model features prominently in these cases. In light of the results of this paper, we strongly suspect a similar approach will work for nondegenerate admissible level $\sQHR{\kk}{\slnpone}{\fsub}$ with $n>2$ and more broadly for nondegenerate admissible level $\sQHR{\kk}{\mathfrak{g}}{\fsub}$ with $\mathfrak{g} \neq \slnpone$.

\subsection{Outline}

We start by discussing the three families of vertex operator algebras central to this paper. The first are subregular W-algebras $\uQHR{\kk}{\mathfrak{g}}{\fsub}$ described in \cref{sec:subreg}. Focusing on type $A_n$ (i.e. $\mathfrak{g} = \slnpone$), a `spectral flow' automorphism of $\uQHR{\kk}{\slnpone}{\fsub}$ is defined and an important class of $\uQHR{\kk}{\slnpone}{\fsub}$-modules is identified. In \cref{sec:reg}, regular W-algebras $\uQHR{\kk}{\mathfrak{g}}{\freg}$ are introduced. The representation theory of $\VOA{W}_{n+1}$ minimal models, which are the rational vertex operator algebras $\sQHR{\kk}{\slnpone}{\freg}$ when  $\kk$ is  nondegenerate-admissible, is described following \cite{AraRat15}. 

The final vertex operator algebra central to this story is the half-lattice vertex algebra $\Pi$ \cite{BerRep02} described in \cref{sec:halflat}. There, we explain how to construct this vertex algebra, choose a conformal structure and define certain ‘relaxed’ $\Pi$-modules that will prove crucial for our investigations.

\cref{sec:invemb} focuses on the relationship between these three families of vertex operator algebras. Unravelling this relationship requires free-field realisations of both $\uQHR{\kk}{\slnpone}{\freg}$ \eqref{eq:ffReg} and $\uQHR{\kk}{\slnpone}{\fsub}$ (Proposition \ref{prop:ffSub}). In \cref{subsec:embedding}, the free-field realisations are used to show that for generic $\kk$, there exists an `inverse quantum hamiltonian reduction' embedding $\uQHR{\kk}{\slnpone}{\fsub} \hookrightarrow \uQHR{\kk}{\slnpone}{\freg} \otimes \Pi$ with an explicitly known screening operator (Theorem \ref{thm:invemb}). This embedding is made even more explicit in \cref{subsec:explicit} where we decompose free-field strong generators of $\uQHR{\kk}{\slnpone}{\fsub}$ \cite{GenStrong20} in terms of fields in $\uQHR{\kk}{\slnpone}{\freg}$ and $\Pi$. In particular, the decompositions describe the embedding for all non-critical $\kk$. This section concludes with some examples.

With the inverse quantum hamiltonian reduction embedding in hand, \cref{sec:relaxedmods} explores the consequences for the representation theory of $\uQHR{\kk}{\slnpone}{\fsub}$. This begins with \cref{subsec:relaxedmod} where we show that taking tensor products of appropriate $\uQHR{\kk}{\slnpone}{\freg}$- and $\Pi$-modules results in ($\ZZ$-graded) $\uQHR{\kk}{\slnpone}{\fsub}$-modules. The characters of these $\uQHR{\kk}{\slnpone}{\fsub}$-modules decompose as products of the characters of the component $\uQHR{\kk}{\slnpone}{\freg}$- and $\Pi$-modules (Corollary \ref{cor:char}). We also prove a number of results regarding when such a tensor product $\uQHR{\kk}{\slnpone}{\fsub}$-module inherits structural properties of its component $\uQHR{\kk}{\slnpone}{\freg}$-module.

Particularly important for logarithmic conformal field theoretic applications are relaxed highest-weight $\uQHR{\kk}{\slnpone}{\fsub}$-modules. Proposition \ref{prop:relaxedmodfromhw} shows that such modules can be constructed as tensor products of irreducible $\uQHR{\kk}{\slnpone}{\freg}$-modules with certain relaxed $\Pi$-modules. 

When the embedding in Theorem \ref{thm:invemb} descends to an embedding of simple quotients $\sQHR{\kk}{\slnpone}{\fsub} \hookrightarrow \sQHR{\kk}{\slnpone}{\freg} \otimes \Pi$ is determined in \cref{sec:simple}. The only admissible levels satisfying these conditions are the nondegenerate ones, at which $\sQHR{\kk}{\slnpone}{\freg}$ is rational. We conclude with the construction of infinitely many relaxed highest-weight $\sQHR{\kk}{\slnpone}{\fsub}$-modules at such levels in \cref{sec:standard}.

The operator product expansions of $\uQHR{\kk}{\mathfrak{g}}{\fsub}$ for $\mathfrak{g} = \sltwo$, $\slthree$ and $\SLA{sl}{4}$ are detailed in Appendix \ref{app:opes} for convenience.

\addtocontents{toc}{\SkipTocEntry}
\subsection*{Acknowledgements}
The author would like to thank David Ridout for countless helpful discussions and feedback on earlier versions of this article. The author would also like to thank Thomas Creutzig for interesting discussions about free-field realisations of W-algebras and bosonisation, and Drazen Adamovi\'c for assistance in dealing with non-generic levels. 
This work is supported by an Australian Government Research Training Program (RTP) Scholarship.

\section{Subregular W-algebras, regular W-algebras and the half-lattice}
\subsection{Subregular W-algebras} \label{sec:subreg}
We begin by defining the main vertex operator algebra of interest and establishing some useful notation. A more detailed account of the strong generators of $\uQHR{\kk}{\slnpone}{\fsub}$ following \cite{GenStrong20} will be deferred to \cref{subsec:embedding}.

\begin{definition}
Let $\mathfrak{g}$ be a simple, finite-dimensional Lie algebra over $\CC$ and $\kk \neq -h^\vee$. The (universal) \emph{subregular W-algebra} $\uQHR{\kk}{\mathfrak{g}}{\fsub}$ is the quantum hamiltonian reduction of the level-$\kk$ universal affine vertex operator algebra $\uaffvoa{\kk}{\mathfrak{g}}$ corresponding to the subregular nilpotent orbit in $\mathfrak{g}$. Denote its unique simple quotient by  $\sQHR{\kk}{\mathfrak{g}}{\fsub}$.
\end{definition}

\noindent While not much is known about subregular W-algebras of type $A$ outside of a few small $n$ examples to be encountered shortly, even less is known about other types. 

The most studied, non-type $A$, non-super example is the type $C_2$ ($\mathfrak{g} = \SLA{sp}{4} \cong \SLA{so}{5}$) subregular W-algebra whose operator product expansions are listed in Section 5 of \cite{Fasquel2020}. The representation theory of $\sQHR{\kk}{ \SLA{sp}{4}}{\fsub}$ at certain levels has been explored \cite{Fasquel2020}, and work on the inverse quantum hamiltonian reduction question has been done from a 4D superconformal field theoretic perspective \cite{Beem21}.

From now on we take $\mathfrak{g} = \slnpone$ (i.e. type $A_n$) and use the notation
\begin{equation}
    \uSubkk = \uQHR{\kk}{\slnpone}{\fsub}, \quad \sSubkk = \sQHR{\kk}{\slnpone}{\fsub}.
\end{equation}
The nilpotent element $\fsub \in \slnpone$ can be taken to be $e_{-\alpha_2}+\dots+e_{-\alpha_n}$, where $\alpha_i$ denotes the $i$'th simple root of $\slnpone$ and $\set{h_i, e_{\alpha_i}, e_{-\alpha_i}}$ is the corresponding $\sltwo$ triple in the Chevalley basis of $\slnpone$. The vertex algebra $\uSubkk$ was long suspected to be isomorphic to the \emph{Feigin--Semikhatov vertex algebra} $\FS{n+1}$ \cite{FS04}. This was proven recently (Theorem 6.9 in \cite{Gen17}) utilising a certain free-field realisation of $\uSubkk$. We will return to the free-field realisation of $\uSubkk$ in \cref{subsec:embedding} as it provides us with convenient explicit formulae for strong generators of $\uSubkk$. 

Certain choices for $n$ in $\uSubkk$ give vertex operator algebras that are well-known. Indeed the main motivation for applying the approach described in this paper to $\uSubkk$ is the success of this approach in small $n$ cases.

\begin{example}[$n=1$]
$\uSub{\kk}{2}$ is the universal affine vertex algebra $\uaffvoa{\kk}{\sltwo}$. It has strong generators denoted by $h(z)$, $e(z)$ and $f(z)$ and their operator product expansions are well-known and are given in Appendix \ref{subsec:A1}. 

That this is an affine vertex algebra and not something more exotic is due to the subregular nilpotent orbit of $\sltwo$ being equal to $\{\left[\begin{smallmatrix} 0&0\\0&0 \end{smallmatrix}\right]\}$. This is the only $n$ for which the affine vertex algebra and subregular W-algebra coincide.
\end{example}

\begin{example}[$n=2$]
$\uSub{\kk}{3}$ is isomorphic to the Bershadsky--Polyakov algebra $\ubpvoak$ \cite{PolGau90, BerCon91} which has strong generators denoted by $J(z)$, $L(z)$, $G^+(z)$ and $G^-(z)$. Corresponding operator product expansions are given in Appendix \ref{subsec:A2}. 

Interestingly the Bershadsky--Polyakov algebra is also isomorphic to the \emph{minimal} W-algebra corresponding to $\slthree$. This helps in, for example, the classification of highest-weight $\ubpvoak$-modules presented in \cite{Feh21} as the general results about minimal W-algebras from \cite{Arakawa2005} apply in this case. This is the only $n$ for which the subregular and minimal nilpotent orbits/W-algebras coincide.
\end{example}

\noindent In general there exists a set of strongly generating fields
\begin{equation} \label{eq:stronggens}
     \left\{G^+(z), J(z), L(z), U_3(z), \dots, U_n(z), G^-(z) \right\} \subset \uSubkk ,
\end{equation}
where we omit the fields $L(z)$, $U_i(z)$ when $n=1$ and omit just the fields $U_i(z)$ when $n=2$. For $n>1$, we can take $L$ to be a conformal vector. The conformal vector of $\uSub{\kk}{2}$ is given by the usual Sugawara construction for $\uaffvoa{\kk}{\sltwo}$. 

That such strongly-generating fields of $\uSubkk$ exist is a consequence of Theorem 4.1 in \cite{KacQua04}. The conformal dimensions with respect to $L$ of the strong generators of $\uSubkk$ above are $1,1,2,3,\dots,n,n$ respectively, and the central charge corresponding to $L(z)$ is
\begin{equation} \label{eq:subcc}
    \Subcck = -\frac{\left( n (\kk+n)-1\right) \left(\kk (n-1) (n^2 + 5 n -2)+(n+1)( n^3 + 3n^2-9n +2)\right)}{(n+1) (\kk+n+1)}.
\end{equation}
As mentioned earlier, the case when $\kk$ is an admissible level for $\slnpone$ is of particular importance for applications to logarithmic conformal field theory. Admissible levels are those $\kk$ satisfying
\begin{equation}
    \kk+n+1 = \frac{\uu}{\vv}, \quad \text{where } \uu \in \ZZ_{\ge n+1}, \vv \in \ZZ_{\ge 1} \text{ and } \text{gcd}\{\uu,\vv \} = 1.
\end{equation} 
At such levels, we use the special notation $\Subminmoduv = \sSubkk$. When $\vv = n$, $(\slnpone,\kk)$ forms an \emph{exceptional pair} \cite{KacWak08}. The corresponding simple vertex operator algebra $\Subminmod{\uu}{n}$ is rational and the modular transformations of characters and fusion rules are in principle known \cite{AvE19}.

Operator product expansions for $\uSubkk$ can be worked out on a case-by-case basis but only a handful are required in what follows. For example,

\begin{equation} \label{eq:subopes}
\begin{aligned}
    J(z) G^\pm(w) \sim \frac{\pm G^\pm(w)}{z-w}, \quad &J(z)J(w) \sim \frac{\ell_n(\kk)\wun}{(z-w)^2}, \quad \text{where  } \ell_n(\kk) = \frac{n \kk}{n+1} + n-1,\\
    L(z) G^\pm(w) \sim \frac{\left(n+1\pm(1-n) \right)G^\pm(w)}{2(z-w)^2} + & \frac{\partial G^\pm(w)}{z-w}, \quad L(z) J(w) \sim \frac{-(n-1)\ell_n(\kk)\wun}{(z-w)^3} + \frac{J(w)}{(z-w)^2} +  \frac{\partial J(w)}{z-w}.
\end{aligned}
\end{equation}
Another important operator product expansion is that between the fields $G^+(z)$ and $G^-(z)$. The other strong generators of $\uSubkk$ in \eqref{eq:stronggens} all appear somewhere in this expansion. That is, $G^+$ and $G^-$ actually generate $\uSubkk$ \cite{FS04}. The complexity of each successive singular term grows rather quickly so we only show the first few terms here. The ellipsis contains all singular terms of order $j < n-1$. 
\begin{align} \label{eq:GpGmope}
    G^+(z)& G^-(w) \sim 
    \frac{\lambda_{n}(n,\kk)\wun}{(z-w)^{n+1}} + \frac{(n+1)\lambda_{n-1}(n,\kk)J(w)}{(z-w)^{n}} \\
    &+ \lambda_{n-2}(n,\kk)\frac{\frac{n}{2}(n+1)\no{JJ}(w)-(\kk+n+1)L(z) + \frac{1}{2}\left((n+1)(n^2-2)+\kk(n+2)(n-1) \right) \partial J(w)}{(z-w)^{n-1}} +\dots, \notag
\end{align}
where
\begin{equation}
    \lambda_{j}(n,\kk) = \prod_{m=1}^j (m(k+n)-1).
\end{equation}
Additional terms of this operator product expansion are presented in Appendix A of \cite{FS04} albeit with respect to a slightly different set of strong generators. To appreciate the complexity of the full set of operator product expansions of $\uSubkk$ for $n>2$, one needs only to refer to the operator product expansions for $\uSub{\kk}{4}$ presented in Appendix \ref{subsec:A3}.

Choosing the conformal structure defined by $L$ has the drawback of making $J(z)$ not quasiprimary and introducing an asymmetry in the conformal dimensions of $G^+$ and $G^-$. This isn't a problem \textit{a priori} but can be rectified by choosing the conformal field to be
\begin{equation} \label{eq:alternativeconformalvec}
    \tilde{L}(z) = L(z)-\frac{n-1}{2} \partial J(z).
\end{equation}
With respect to $\tilde{L}$, $G^+$ and $G^-$ have conformal dimension $(n+1)/2$ and $J(z)$ is a primary field of conformal dimension 1. The $\tilde{L}$ conformal structure has the drawback of requiring the consideration of twisted modules when $n$ is even owing to the half-integer conformal dimension of $G^\pm$:

An important extension of the work presented here is to consider the modularity of conjectured standard modules and compute their Grothendieck fusion rules. The presence of twisted modules complicates such computations as care must be given to which sector one is working in. Therefore unless otherwise indicated we will keep $L$ as the conformal vector. 

Another reason for sidestepping twisted modules is the existence of a \textit{spectral flow} automorphism of $\uSubkk$ that exchanges twisted and untwisted sectors when $n$ is even. To construct this automorphism, we expand homogeneous fields of $\uSubkk$ as
\begin{equation}
    A(z) = \sum_{m \in \ZZ} A_{(m)} z^{-m-1} = \sum_{m \in \ZZ} A_m z^{-m-\Delta_A}
\end{equation}
where $\Delta_A$ is the conformal dimension of $A$ with respect to $L(z)$. The spectral flow automorphism is constructed using certain intertwining operators \cite{LiPhy97}.

\begin{proposition}
Let $\ell \in \ZZ$. The map $\sfsymb^\ell:\uSubkk \rightarrow \uSubkk$ defined by
\begin{equation} \label{eq:sfdef}
    \sfmod{\ell}{A(z)} = Y\left(\Lambda(\ell J,z)A, z\right), \quad \text{where  } \Lambda(\ell J,z) = z^{-\ell J_0} \prod_{m=1}^\infty \emph{exp} \left( \frac{(-1)^m}{m} \ell J_m z^{-m} \right).
\end{equation}
is a vertex algebra automorphism, where $Y$ is the vertex map for $\uSubkk$.
\end{proposition}

\noindent That this a vertex algebra automorphism is a straightforward application of Proposition 3.2 in \cite{LiPhy97}. Direct computation shows that 
\begin{equation} \label{eq:sfexamples}
\begin{gathered}
    \sfmod{\ell}{G^\pm(z)} = z^{\mp \ell} G^{\pm}(z), \quad \sfmod{\ell}{J(z)} = J(z) - \ell_n(\kk) \ell z^{-1}, \\
    \sfmod{\ell}{\tilde{L}(z)} = \tilde{L}(z) - \ell z^{-1} J(z) + \frac{1}{2}\ell_n(\kk) \ell^2 z^{-2}.
\end{gathered}
\end{equation}
The action of spectral flow can also be written in terms of modes as follows:
\begin{equation}
\begin{gathered}
    \sfmod{\ell}{G^\pm_m} = G^\pm_{m \mp \ell}, \quad 
    \sfmod{\ell}{J_m} = J_m - \ell_n(\kk) \ell \delta_{m,0} \wun, \\
    \sfmod{\ell}{\tilde{L}_m} = \tilde{L}_m - \ell J_m + \frac{1}{2}\ell_n(\kk) \ell^2 \delta_{m,0} \wun.
\end{gathered}
\end{equation}
In principle, for any fixed $n$, one could also compute the action of spectral flow on the the fields $U_3(z), \dots, U_n(z)$ given complete information about the relevant operator product expansions. In the first case where one gets such fields ($n=3$), spectral flow acts on the field $U_3(z)$ as the identity automorphism. However, for general $n$ where the operator product expansions are more involved, the action of spectral flow on the fields $U_i(z)$ is more difficult to determine. 

What can be shown using the free-field expansions described in \cref{subsec:explicit} is that the fields $\set{U_3(z), \dots, U_n(z)}$ have `$J$-charge' 0, i.e. that the $J_0 U_i = 0$ for all $i$. Therefore by \eqref{eq:sfdef}, spectral flow acts on the modes of the form $(U_i)_m$ as
\begin{equation}
    \sfmod{\ell}{(U_i)_m} = (U_i)_m + \dots.
\end{equation}
As the characters we will eventually define here for $\uSubkk$-modules only keep track of $J_0$- and $L_0$-eigenvalues, the formulae \eqref{eq:sfexamples} and the fact that the $U_i(z)$ have $J$-charge 0 is sufficient for our purposes. 

From the definition of spectral flow \eqref{eq:sfdef} it is clear that the inverse of $\sfsymb^\ell$ is $\sfsymb^{-\ell}$. Moreover, spectral flow is only a vertex operator algebra automorphism for $\ell = 0$ as
\begin{equation} \label{eq:spectralflowofL}
    \sfmod{\ell}{L(z)} = \sfmod{\ell}{\tilde{L}(z)}+\frac{n-1}{2}\partial \sfmod{\ell}{J(z)} = L(z) - \ell z^{-1} J(z) + \frac{1}{2}\ell_n(\kk)\left( \ell^2 +  \ell (n-1) \right) z^{-2}.
\end{equation}
Given a $\uSubkk$-module $\mathcal{M}$ and an automorphism $\omega$ of $\uSubkk$, we can define a new $\uSubkk$-module $\omega^*(\mathcal{M})$ (where $\omega^*$ is an arbitrary vector space isomorphism) by the action
\begin{equation}
    A(z) \cdot \omega^*(v) = \omega^*(\omega^{-1}\left(A(z)\right)v), \quad \text{where } v \in \mathcal{M}, \ A(z) \in \uSubkk.
\end{equation}
Since this action doesn't depend on the choice of $\omega^*$ up to vertex algebra module isomorphisms, we'll drop the star and denote  $\omega^*(\mathcal{M})$ by  $\omega(\mathcal{M})$.

As in the $n=1$ and $2$ cases, we restrict attention to a particular subclass of $\uSubkk$-modules called \emph{weight modules}. The corresponding category $\categ{W}_{\kk}$ of weight modules for the simple quotient $\sSubkk$ is expected to have the modular properties desirable for defining a logarithmic conformal field theory for certain $\kk$. 

To define weight modules, let $\mathsf{U}$ be the mode algebra of $\uSubkk$. That is, $\mathsf{U}$ is the unital associative $\CC$-algebra spanned by the modes $A_n$ for $A(z) \in \uSubkk$ subject to the commutation relations defined by the operator product expansions. The grading on $\mathsf{U}$ by $[L_0, \cdot]$-eigenvalue gives a generalised triangular decomposition \cite{KacQua04}
\begin{equation}
    \mathsf{U} = \mathsf{U}_> \otimes \mathsf{U}_0 \otimes \mathsf{U}_<,
\end{equation}
where $\mathsf{U}_>$, $\mathsf{U}_0$ and $\mathsf{U}_<$ denote the unital subalgebras generated by $A_m$ for all homogeneous $A(z) \in \uSubkk$ with $m>0$, $m=0$ and $m>0$ respectively.

\begin{definition}
\begin{itemize}
\item A vector $v$ in a $\sSubkk$-module $\mathcal{M}$ is a \emph{weight vector} of \emph{weight} $(j,\Delta)$ where $j$, $\Delta \in \CC$ if it is a simultaneous eigenvector of $J_0$ and $L_0$ with eigenvalues $j$ (\emph{charge)} and $\Delta$ (\emph{conformal dimension}) respectively. The nonzero simultaneous eigenspaces of $J_0$ and $L_0$ are called \emph{weight spaces} of $\mathcal{M}$. If $\mathcal{M}$ has a basis of weight vectors and each weight space is finite-dimensional, then $\mathcal{M}$ is a \emph{weight module}.
\item A vector in a $\uSubkk$-module is a \emph{highest-weight vector} if it is a weight vector that is annihilated by the action of $\mathsf{U}_>$ and $G^+_0$. A \emph{highest-weight module} is a $\uSubkk$-module generated by a highest-weight vector.
\item A vector in a $\uSubkk$-module is a \emph{relaxed highest-weight vector} if it is a weight vector that is annihilated by the action of $\mathsf{U}_>$. A \emph{relaxed highest-weight module} is a $\uSubkk$-module generated by a relaxed highest-weight vector.
\end{itemize}
\end{definition}

\noindent The definition of conjugate highest-weight vectors and modules is identical except with $G^+_0$ replaced with $G^-_0$. From the actions of the spectral flow automorphism, we see that if $v \in \mathcal{M}$ is a weight vector of charge $j$ and conformal dimension $\Delta$, $\sfmod{\ell}{v} \in \sfmod{\ell}{\mathcal{M}}$ is a weight vector with charge and conformal dimension
\begin{equation}
    j' = j + \ell_n(\kk) \ell, \quad \Delta' = \Delta + j \ell + \frac{1}{2}\ell_n(\kk) \left(\ell^2 -\ell(n-1)\right)
\end{equation}
respectively.

\subsection{Regular W-algebras and their modules} \label{sec:reg}
By far the most studied and well-understood W-algebras are the \emph{regular} (or \emph{principal}) W-algebras $\uQHR{\kk}{\mathfrak{g}}{\freg}$. There is much one can say about these vertex operator algebras (see for example the reviews \cite{BouSom91,BouWSym93}). For the purpose of gaining insights into subregular W-algebras there are two main features at play:

One is the partial ordering on quantum hamiltonian reductions of $\uaffvoa{\kk}{\mathfrak{g}}$ induced by an ordering on nilpotent orbits \cite{Collingwood17}. There is evidence that in addition to the usual quantum hamiltonian reduction, there is a `partial quantum hamiltonian reduction' between two W-algebras related by the partial ordering. In particular, the regular nilpotent orbit is strictly greater than the subregular nilpotent orbit for all $\mathfrak{g}$.

The philosophy adopted here is that partial quantum hamiltonian reductions can be inverted in the sense described by Adamovi\'c and others \cite{AdaRea17, AdaRea20}. In light of the observation about subregular and regular nilpotent orbits, this manifests concretely as the subregular W-algebra being able to be embedded into the regular W-algebra tensored with some lattice vertex algebra.

The other main feature of regular W-algebras is that the simple regular W-algebra  $\sQHR{\kk}{\mathfrak{g}}{\freg}$ is rational \cite{AraRat15} for almost all admissible levels. These are particularly interesting levels from a logarithmic conformal field theory point of view.

\begin{definition}
Let $\mathfrak{g}$ be a simple, finite-dimensional Lie algebra over $\CC$ and $\kk \neq -h^\vee$. The  (universal) \emph{regular W-algebra} $\uQHR{\kk}{\mathfrak{g}}{\freg}$ is defined as the quantum hamiltonian reduction of the level-$\kk$ universal affine vertex operator algebra $\uaffvoa{\kk}{\mathfrak{g}}$ corresponding to the regular nilpotent orbit in $\mathfrak{g}$. Denote its unique simple quotient by $\sQHR{\kk}{\mathfrak{g}}{\freg}$.
\end{definition}

\noindent As before, we let $\mathfrak{g} = \slnpone$and use the notation
\begin{equation}
    \uRegkk = \uQHR{\kk}{\mathfrak{g}}{\freg}, \quad \sRegkk = \sQHR{\kk}{\mathfrak{g}}{\freg}.
\end{equation}
In this case, the nilpotent element $\freg \in \slnpone$ can be taken to be $e_{-\alpha_1}+e_{-\alpha_2}+\dots+e_{-\alpha_n}$. This W-algebra was first defined in the $\mathfrak{g} = \slthree$ case \cite{ZamInf85} and later for $\mathfrak{g} = \slnpone$ \cite{FatLuk88}. There, the regular W-algebra is described as the intersections of kernels of certain screening operators: Let $\VOA{H}_\alpha$ be the Heisenberg vertex algebra strongly-generated by $n$ fields $\alpha_1(z), \dots, \alpha_n(z)$ with operator product expansion
\begin{equation}
    \alpha_i(z) \alpha_j(w) \sim \frac{A_{i,j}(\kk+n+1) \mathbbm{1}}{(z-w)^2},
\end{equation}
where $[A]$ is the Cartan matrix of $\slnpone$. For non-critical $\kk$ (i.e. $\kk \neq -n-1$), there exists a vertex operator algebra embedding $\uRegkk \hookrightarrow \VOA{H}_\alpha$ whose image is determined by
\begin{equation} \label{eq:ffReg}
    \uRegkk \cong \bigcap_{i=1}^n \text{ ker} \int \ee^{\frac{-1}{\kk+n+1}\alpha_i}(z) \ dz.
\end{equation}
when $\kk$ is generic. Such an embedding (of vertex algebras) also exists for critical level. This identification gives a useful free-field description of strong generators for the regular W-algebra using the \emph{Miura transform}. Following the presentation in \cite[Section.~6.3.3]{BouWSym93}, let $\varepsilon_s$, $s = 1,\dots, n+1$ denote the weights of the defining representation of $\slnpone$ normalised so that $\alpha_i = \varepsilon_i-\varepsilon_{i+1}$ for all simple roots $\alpha_i$ of $\slnpone$, and $\sum_s \varepsilon_s = 0$. The relationship between the simple roots $\alpha_i$ and weights $\varepsilon_s$ can be inverted and upgraded to a relationship between fields given by
\begin{equation}
    \varepsilon_s(z) = -\sum_{j=1}^{s-1}\frac{j}{n+1}\alpha_j(z) + \sum_{j=s}^{n} \frac{n+1-j}{n+1} \alpha_j(z).
\end{equation}
Define the generating function of a set of fields $\set{W_0(z), W_1(z) , \dots, W_{n+1}(z)} \subset \VOA{H}_\alpha$ by
\begin{equation}
    R_n(z) = -\sum_{s=0}^{n+1} W_s(z) \left( (\kk+n) \partial  \right)^{n+1-s} 
    = \left( (k+n) \partial - \varepsilon_{n+1}(z) \right) \dots \left( (k+n) \partial - \varepsilon_1(z) \right).
\end{equation}
It can be shown that the operator product expansion of $R_n(z)$ and the screening operators in \eqref{eq:ffReg} is a total derivative so the component fields $W_s(z)$ are all $\uRegkk$ fields. Additionally, the fields $W_2(z), \dots, W_{n+1}(z)$ strongly-generate $\uRegkk$ for all $\kk$ \cite{FatLuk88}. These fields are not in general quasi-primary but one can usually take appropriate linear combinations of them and their derivatives to obtain primary fields. Fortunately this is not necessary for our purposes.

To find convenient closed-form expressions for the fields $W_s(z)$ in terms of the fields of $\VOA{H}_\alpha$, we recall the noncommutative elementary symmetric polynomials as described in \cite[Ch.~12]{Molev18}.

\begin{definition}
Let $\omega_1, \dots, \omega_N$ be $N$ mutually noncommutative operators. The \emph{m-th noncommutative elementary symmetric polynomial} in $\omega_1, \dots, \omega_N$ is
\begin{equation}
    E_m(\omega_1, \dots, \omega_N) = \sum_{i_1 > \dots > i_m} \omega_{i_1} \dots \omega_{i_m}.
\end{equation}
\end{definition}

\noindent Rewriting $R_n(z)$ in terms of modes, we therefore have:
\begin{equation} \label{eq: miurabasis}
\begin{aligned}
    \big( (k+n) \partial - (\varepsilon_{n+1})_{-1} \big) \dots \big( (k+n) \partial - (\varepsilon_1)_{-1} \big) =& \ E_{n+1}\big( (k+n) \partial - (\varepsilon_1)_{-1}, \dots, (k+n) \partial - (\varepsilon_{n+1})_{-1}  \big)\\
    =& \  -\sum_{s=0}^{n+1} (W_s)_{(-1)} \left( (\kk+n) \partial  \right)^{n+1-s}
\end{aligned}
\end{equation} 
By Proposition 12.4.4 of \cite{Molev18}, 
\begin{equation} \label{eq:freefieldforregfields}
    W_s = -E_{s}\big( (k+n) \partial - (\varepsilon_1)_{-1}, \dots, (k+n) \partial - (\varepsilon_{n+1})_{-1}  \big) \mathbbm{1}
\end{equation}
where $\mathbbm{1}$ denotes the vacuum of $\VOA{H}_\alpha$. For example, $W_0(z) = -\mathbbm{1}(z)$, $W_1(z) = 0$ and
\begin{equation}
    W_2(z) 
    = (k+n)\sum_{j=1}^{n+1} (n+1-j) \partial \varepsilon_{j}(z) - \sum_{i > j} \no{\varepsilon_{i}(z) \varepsilon_{j}(z)}.
\end{equation}
Moreover, $T(z) = \frac{1}{\kk+n+1} W_2(z)$ is an energy-momentum field with central charge
\begin{equation}
    \Regcck = -\frac{n\left( (n+1)(\kk-1)+n^2 + 2n \right)\left((n+2)\kk + (n+1)^2\right)}{\kk+n+1}.
\end{equation}
The generating fields $W_i(z)$ have conformal dimension $i$ with repect to $T(z)$. It is known that the vertex operator algebra $\uRegkk$ is reducible if $\kk$ is a \emph{non-degenerate admissible level} \cite{Ara07}. That is, 
\begin{equation}
    \kk + n+1 = \frac{\uu}{\vv}, \quad \text{where } \uu, \vv \in \ZZ_{\ge n+1} \text{ and } \text{gcd}\{\uu,\vv \} = 1.
\end{equation} 
At these levels, $\sRegkk$ is rational \cite{AraRat15}. For this reason, we use the special notation $\Regminmoduv = \sRegkk$ when $\kk$ is nondegenerate-admissible. The modules of $\Regminmoduv$ are all highest-weight modules and admit a parametrisation in terms of $\slnpone$ weights.

Following Section 8.3 in \cite{AvE19a}, let $\text{Pr}^k$ be the set of principal admissible $\slnpone$ weights of level $\kk$. Each weight $\lambda \in \text{Pr}^k$ defines a central character $\gamma_\lambda : Z(\slnpone)\rightarrow \CC$ by evaluation. Let $\text{Pr}^k_\mathcal{W}$ be the set of all such central characters with $\lambda$ ranging over $\text{Pr}^k$. Associated to each $\gamma \in \text{Pr}^k_\mathcal{W}$ is a simple $\Regminmoduv$-module $\mathcal{W}_\gamma$. Simple $\Regminmoduv$-modules have been classified \cite{AraRat15} and consist only of the set of modules
\begin{equation}
	\{ \mathcal{W}_\gamma \ \vert \ \gamma \in \text{Pr}^k_\mathcal{W} \}.
\end{equation}
All the modules above are highest-weight with one-dimensional top spaces and are mutually non-isomorphic. Let $v_\gamma$ be the highest-weight vector of  $\mathcal{W}_\gamma$. Without loss of generality and owing to the Miura basis fields being strong generators of $\uRegkk$, we can view $\gamma$ as an element of $\CC^{n}$ with $\gamma \leftrightarrow (\gamma_2, \dots, \gamma_{n+1})$ defined by
\begin{equation}
    (W_i)_0 \ v_\gamma = \gamma_i v_\gamma.
\end{equation}
For example, the vacuum module is $\Regminmoduv \cong \mathcal{W}_{0}$.

\subsection{The half-lattice vertex operator algebra and its modules} \label{sec:halflat}
To describe the relationship between $\uRegkk$ and $\uSubkk$, we will also need the ``half-lattice'' vertex operator algebra $\Pi$ \cite{BerRep02}. This is the same vertex operator algebra used in the inverse quantum hamiltonian reduction relating $\uWvoak$ and $\ubpvoak$ \cite{AdaRea20} and in the inverse quantum hamiltonian reduction relating $\VOA{Vir}^\kk$ and $\uaffvoa{\kk}{\sltwo}$ \cite{AdaRea17}. 

As these are simply the $n=2$ and $n=1$ cases of the regular-subregular inverse reduction, our description of $\Pi$ closely follows that given in these cases (\cite[Sec.~3]{AdaRea20} and \cite[Sec.~4]{AdaRea17}).

Consider the abelian Lie algebra $\alg{h} = \text{span}_{\CC} \set{c,d}$, equipped with the symmetric bilinear form $\inner{\cdot\,}{\cdot}$ defined by
\begin{equation}
	\inner{c}{c} = \inner{d}{d} = 0 \quad \text{and} \quad \inner{c}{d} = 2.
\end{equation}
The group algebra $\CC[\ZZ c] = \text{span}_{\CC} \{e^{nc} \vert\ n \in \ZZ \}$ has the structure of an $\alg{h}$-module according to the formula
\begin{equation}
    h(e^{nc})= n\langle h, c \rangle e^{nc}.
\end{equation}
Denote by $\VOA{H}$ the Heisenberg vertex algebra defined by $\alg{h}$ and $\inner{\cdot\,}{\cdot}$.

\begin{definition} \label{def:halflattice}
The \emph{half lattice vertex algebra} $\halflattice$ is the lattice vertex algebra $\VOA{H} \otimes \CC[\ZZ c]$ where the action of $h \in \alg{h}$ on $\CC[\ZZ c]$ is identified with the action of the zero mode $h_0$ of $h(z) \in \VOA{H}$.
\end{definition}

\noindent A set of (strong) generating fields for $\halflattice$ is $\set*{c(z), d(z), \ee^{mc}(z) \st m \in \ZZ}$.  The \opes\ of these fields are easily determined:
\begin{equation} \label{eq:ope:halflattice}
\begin{aligned}
		c(z)c(w) &\sim 0, & c(z)d(w) &\sim \frac{2 \, \wun}{(z-w)^2}, & d(z)d(w) &\sim 0, \\
		c(z)\ee^{mc}(w) &\sim 0, & d(z)\ee^{mc}(w) &\sim \frac{2m \, \ee^{mc}(w)}{z-w}, & \ee^{mc}(z) \ee^{nc}(w) &\sim  0.
\end{aligned}
\end{equation}
We will often use the convenient orthogonal basis for the Heisenberg fields in $\halflattice$ given by
\begin{equation} \label{eq:abbasis}
	a(z) = -\frac{\ell_n(\kk)}{2} c(z) + \frac{1}{2}d(z) \quad \text{and} \quad b(z) = \frac{\ell_n(\kk)}{2} c(z) + \frac{1}{2}d(z),
\end{equation}
Note that $-\inner{a}{a} =\inner{b}{b} = \ell_n(\kk)$, while $\inner{a}{b}=0$.

The half lattice vertex algebra admits a two-parameter family of energy-momentum fields given by
\begin{equation} \label{eq:piemfield}
	t(z) = \frac{1}{2} \no{c(z)d(z)} + \alpha \pd c(z) + \beta \pd d(z), \quad \alpha, \beta \in \CC;
\end{equation}
with corresponding central charge $2-48\alpha \beta$.  We equip $\halflattice$ with the conformal structure given by \eqref{eq:piemfield} with $\alpha = \frac{n}{2} \ell_n(\kk)$ and $\beta = \frac{-1}{2}$, so that $t(z) = \frac{1}{2} \no{c(z)d(z)} - \frac{n+1}{2} \pd a(z) + \frac{n-1}{2} \pd b(z)$. The associated central charge $\lcck$ satisfies
\begin{equation}
    \lcck + \Regcck = \Subcck.
\end{equation}
At the admissible levels we are interested in, the central charge of $\halflattice$ simplifies to
\begin{equation} \label{eq:ccaddup}
	\lccuv = 2 + 12 n \ell_n(\kk) = \frac{2 \left(6 n^2 (\uu-\vv)-5 n \vv+\vv\right)}{(n+1) \vv}.
\end{equation}
With respect to $t(z)$, both $a(z)$ and $b(z)$ have conformal dimension $1$ (though both are not quasiprimary) whilst that of $\ee^{mc}(z)$ is $m$. As for $\uSubkk$ we can use the Heisenberg field $b(z)$ to define an vertex algebra automorphism of $\Pi$ called \emph{spectral flow} $\rho^\ell$ for all $\ell \in \ZZ$.

\begin{proposition}
Let $\ell \in \ZZ$. The map $\rho^\ell:\Pi \rightarrow \Pi$ defined by
\begin{equation} \label{eq:sfdefPi}
    \rho^\ell \left( A(z)\right) = Y\left(\Lambda(\ell b,z)A, z\right), \quad \text{where  } \Lambda(\ell b,z) = z^{-\ell b_0} \prod_{m=1}^\infty \emph{exp} \left( \frac{(-1)^m}{m} \ell b_m z^{-m} \right).
\end{equation}
is a vertex algebra automorphism, where $Y$ is the vertex map for $\Pi$.
\end{proposition}

\noindent The action of spectral flow on the fields $a(z)$, $b(z)$ and $\ee^{mc}(z)$ where $m \in \ZZ$ can be explicitly computed with help from \eqref{eq:ope:halflattice} and is given by
\begin{equation}
    \rho^\ell(a(z)) = a(z),\quad \rho^\ell(b(z)) = b(z) - \ell_n(\kk) \ell z^{-1},\quad \rho^\ell(\ee^{mc}(z)) = z^{-m \ell} \ee^{mc}(z).
\end{equation}
This is only a vertex operator algebra automorphism when the conformal vector $t(z)$ is preserved i.e. when $\ell = 0$. As seen with spectral flow for $\uSubkk$, for any $\Pi$-module and any $\ell \in \ZZ$ we get a new $\Pi$-module $\rho^\ell (M)$. 

Here as before, the range of $\ell$ can be extended to include $\ZZ+\frac{1}{2}$ in which case a $\ZZ$-graded $\Pi$-module becomes $\ZZ+\frac{1}{2}$-graded upon twisting with respect to $\rho^\ell$ and vice-versa.

We are interested in the positive-energy weight modules of $\halflattice$, meaning those on which the $h_0$, with $h \in \alg{h}$, act semisimply and $t_0$ has eigenvalues that are bounded below.  Such modules can be induced from the $\ZZ c$-modules generated by (certain) elements $\ee^h \in \CC[\alg{h}]$ on which $h' \in \alg{h}$ acts as $h' \cdot \ee^h = \inner{h'}{h} \, \ee^h$ \cite{BerRep02}.  The following is adapted from \cite{AdaRea17, AdaRea20} to accommodate our choice of conformal structure.

\begin{proposition} \label{prop:lvoamod}
The weight $\halflattice$-module generated from $\ee^{rb + \lambda c}$ is positive-energy if and only if $r=-1$.  In this case, the $\halflattice$-module is $\ZZ$-graded, simple and the minimal $t_0$-eigenvalue is $\frac{n}{2} \ell_n(\kappa)$.
\end{proposition}

\noindent Denote the $\Pi$-module generated from $\ee^{rb + \lambda c}$ by $\halflatticemod{r}{\lambda}$. These $\Pi$-modules satisfy a number of nice properties: By direct computation, the eigenvalue of $b_0$ on $\ee^{-b + \lambda c}$ is $\lambda - \ell_n(\kk)$. The zero modes $\ee^{\pm c}_0$ act injectively on the top space of $\halflatticemod{-1}{\lambda}$. In general, $\halflatticemod{r}{\lambda}$ is $\ZZ$-graded as long as $r \in \ZZ$ and $\ZZ+\frac{1}{2}$-graded when  $r \in \ZZ +\frac{1}{2}$. In either case, $\halflatticemod{r}{\lambda} \cong \halflatticemod{r}{\lambda+n}$ as $\Pi$-modules for all $n \in \ZZ$. 

Finally, due to the following proposition, we can view $r$ as a spectral flow index.

\begin{proposition}
Let $r,\ell \in \frac{1}{2}\ZZ$ and $\lambda \in \CC$.
\begin{equation}
    \rho^\ell\left(\halflatticemod{r}{\lambda} \right) \cong \halflatticemod{r+\ell}{\lambda}.
\end{equation}
\end{proposition}

\noindent Spectral flow of the positive-energy module $\halflatticemod{-1}{\lambda}$ is only positive-energy if $\ell = 0$, in which case spectral flow coincides with the identity automorphism. As $\halflatticemod{r}{\lambda}$ are all relaxed modules for a lattice-like vertex algebra, their characters are straightforward to compute. Here, we define the character of a $\Pi$-module $M$ to be
\begin{equation}
    \charac{M}{z}{q} = \text{tr}_M \left(z^{b_0} q^{t_0 -\frac{\lcck}{24}} \right).
\end{equation}
A computation using a PBW basis for $\Pi$ shows that
\begin{equation} \label{eq:minusonelatticemodcharac}
    \charac{\halflatticemod{-1}{\lambda}}{z}{q} = \frac{z^{\lambda - \ell_n(\kk)}}{\eta(q)^2} \sum_{i \in \ZZ} z^i,
\end{equation}
which also gives the characters of $\halflatticemod{r}{\lambda}$ for all $r \in \frac{1}{2}\ZZ$ as
\begin{equation} \label{eq:rlatticemodcharac}
\begin{aligned}
    \charac{\halflatticemod{r}{\lambda}}{z}{q} &= \text{tr}_{\rho^{r+1}\left( \halflatticemod{-1}{\lambda} \right)} \left(z^{b_0} q^{t_0 -\frac{\lcck}{24}} \right)\\
    &= \text{tr}_{\halflatticemod{-1}{\lambda}} \left(z^{\rho^{-r-1}(b_0)} q^{\rho^{-r-1}(t_0) -\frac{\lcck}{24}} \right) \\
    &= z^{(r+1)\ell_n(\kk)}q^{(r+1)(r+2-n)\frac{\ell_n(\kk)}{2}}\charac{\halflatticemod{-1}{\lambda}}{z q^{r+1}}{q}.
\end{aligned}
\end{equation}

\section{Inverse quantum hamiltonian reduction} \label{sec:invemb}
An embedding $\uSubkk \hookrightarrow \uRegkk \otimes \Pi$ for $n=1$ and $2$ can be obtained by fairly direct methods; In these cases, the operator product expansions are known on both sides and are straightforward to work with. The desired embedding can be obtained by making some reasonable assumptions (the embedding is conformal, the Heisenberg field gets mapped to a linear combination of $c(z)$ and $d(z)$, \dots) and imposing the operator product expansions of $\uSubkk$. 

Once this is done, the work of showing that the image of $\uSubkk$ is actually $\uSubkk$ and not some quotient remains. In both the $n=1$ and $2$ cases, this can be shown directly using suitable bases of the vertex algebras involved.

When $n>2$, the complexity of the operator product expansions makes this approach exceedingly difficult. Therefore to prove the existence of this embedding in general, we need some more information about $\uSubkk$. This information takes the form of a free-field realisation of $\uSubkk$ obtained as the kernel of certain screening operators \cite{Gen17}.

An additional benefit of the free-field approach is that the map $\uSubkk \rightarrow \uRegkk \otimes \Pi$ we get is automatically injective. This is because the map is a composition of the free-field realisation of $\uSubkk$ (which is injective \cite{Gen17}), the FMS bosonisation of the $\beta \gamma$ ghost system vertex algebra (which is injective \cite{FMS86}) and a vertex algebra isomorphism.

\subsection{An embedding} \label{subsec:embedding} 

Let $\VOA{B}$ be the $\beta \gamma$ ghost system vertex algebra (also known as the Weyl vertex algebra). This vertex algebra is strongly generated by two bosonic fields $\beta(z)$ and $\gamma(z)$ with operator product expansions
\begin{equation}
    \beta(z) \beta(w) \sim 0 \sim \gamma(z) \gamma(w), \quad \beta(z) \gamma(w) \sim \frac{\mathbbm{-1}}{z-w}.
\end{equation} 
 To minimise notational clutter, we will suppress the tensor product symbol when it is clear on which vertex algebra the involved fields are acting on. We will also suppress tensor products involving vacuum fields when possible.
\begin{proposition}[Theorem 3.2 \cite{CreGen21}] \label{prop:ffSub}
Let $\kk \neq -n-1$. The subregular W-algebra $\uSubkk$ embeds into $\VOA{H}_\alpha \otimes \VOA{B}$ and the image of this embedding for generic $\kk$ is:
\begin{equation} \label{eq:ffSub}
    \uSubkk \cong \left( \text{ \emph{ker}} \int \beta(z) \ee^{\frac{-1}{\kk+n+1}\alpha_1}(z) \ dz \right) \cap \left( \bigcap_{i=2}^n \text{ \emph{ker}} \int \ee^{\frac{-1}{\kk+n+1}\alpha_i}(z) \ dz \right) \subset \VOA{H}_\alpha \otimes \VOA{B}.
\end{equation}

\end{proposition}

\noindent As for regular W-algebras, a similar vertex algebra embedding exists for critical $\kk$. The free-field realisation \eqref{eq:ffReg} of $\uRegkk \hookrightarrow \VOA{H}_\alpha$ has the same screening operators as those for $\uSubkk$ except for the one involving the $\alpha_1(z)$ vertex operator. To be able to recognise the $\uRegkk$ screening operators in \eqref{eq:ffSub}, we therefore need to deal with the $\beta(z)$ appearing in the first screening operator. 

It is conceivable that there is some chain complex involving $\uSubkk$ with associated differential whereupon taking the zeroth homology has the effect of setting $\beta(z) = 1$, reminiscent of the usual quantum hamiltonian reduction of affine vertex algebras \cite{KacQua03}. Indeed this is exactly what one would expect a partial quantum hamiltonian reduction from $\uSubkk$ to $\uRegkk$ to do. Our interest here is the `inverse' of this as-of-yet undefined partial reduction.

We would like to tackle the $\beta(z)$ factor by `absorbing' it into the vertex operator it sits next to in the free-field-realisation of $\uSubkk$. In order to do this, we use the \emph{Friedan-Martinec-Shenker (FMS) bosonisation} of $\VOA{B}$ \cite{FMS86}.
\begin{proposition} \label{prop:FMS}
The vertex algebra homomorphism $\varphi:\VOA{B} \rightarrow \Pi$ defined by
 \begin{equation}
    \beta(z) \mapsto \ee^{c}(z), \ 
    \gamma(z) \mapsto \frac{1}{2}\no{\left(c(z)+d(z)\right) \ee^{-c}(z)}
\end{equation}
is an embedding whose image is specified by
\begin{equation} \label{eq:FMSscreening}
    \VOA{B} \cong \text{ \emph{ker}} \int \ee^{\frac{1}{2}c + \frac{1}{2}d}(z) \ dz.
\end{equation}
\end{proposition}
\noindent We can therefore compose the embedding $\uSubkk \hookrightarrow \VOA{H}_\alpha \otimes \VOA{B}$ with FMS bosonisation $\varphi$ tensored with the identity map $\id_{\VOA{H}_\alpha}$ to obtain an embedding $\uSubkk \hookrightarrow \VOA{H}_\alpha \otimes \Pi$. To describe this embedding using screening operators, let $S(z)$ be one of the screening fields in \eqref{eq:ffSub}. Let $F_S$ be the Fock space of $\VOA{H}_\alpha$ mapped to by $\int S(z) \ dz$. That is,
\begin{equation}
    \int S(z) \ dz: \VOA{H}_\alpha \otimes \VOA{B} \rightarrow F_S \otimes \VOA{B}.
\end{equation}
What we would like is a screening operator $\int S'(z) \ dz: \VOA{H}_\alpha \otimes \Pi \rightarrow F_S \otimes \Pi$ such the following diagram commutes:
\begin{equation} \label{eq:diagram}
\begin{tikzpicture}[xscale=2.5,yscale=1.5,baseline=(current  bounding  box.center)]
		\draw[right hook->]
		(0,2)  node[above] {$\VOA{H}_\alpha \otimes \VOA{B}$} -- 
		(0,1) node[below] {$\VOA{H}_\alpha \otimes \Pi$} node[midway, left] {$\id_{\VOA{H}_\alpha}\otimes \ \varphi$};
		\draw[->]
		(0.3,2.2)  -- 
		(2,2.2) node[right] {$F_S \otimes \VOA{B}$} node[midway, above] {$\int S(z) \ dz$};
		\draw[dashed,->]
		(0.3,0.8)  -- 
		(2,0.8) node[right] {$F_{S} \otimes \Pi$} node[midway, above] {$\int S'(z) \ dz$};
		\draw[right hook->]
		(2.3,2)  -- 
		(2.3,1)  node[midway, right] {$\id_{F_S}\otimes \ \varphi$};
\end{tikzpicture}
\end{equation}
If the screening operator $\int S'(z) \ dz$ exists, it can be shown that
\begin{equation}
    \id_{\VOA{H}_\alpha}\otimes \ \varphi \left( \text{ker}\left(\int S(z) \ dz\right) \right) = 
    \text{ker}\left(\int S'(z) \ dz \right) \cap  \text{im}\left( \id_{\VOA{H}_\alpha}\otimes \ \varphi \right).
\end{equation}
For $S(z) = \ee^{-1/(\kk+n+1)\alpha_i}(z)$ with $i = 2,\dots, n$, we can take $S'(z) = S(z)$ as $S(z)$ doesn't act on $\VOA{B}$. For $S(z) = \beta(z)\ee^{-1/(\kk+n+1)\alpha_1}(z)$, since all arrows in \eqref{eq:diagram} are vertex algebra homomorphisms it is sufficient to replace $\beta(z)$ with its image under $\varphi$ to obtain $S'(z) = \ee^{c}(z)\ee^{-1/(\kk+n+1)\alpha_1}(z)$. Therefore the image of the embedding $\uSubkk \hookrightarrow \VOA{H}_\alpha \otimes \Pi$ is specified by
\begin{equation} \label{eq: subreglatticeemb}
\begin{aligned}
    \uSubkk 
    &\cong \left( \text{ ker} \int \ee^{c}(z) \ee^{\frac{-1}{\kk+n+1}\alpha_1}(z) \ dz \right) \cap \left( \bigcap_{i=2}^n \text{ ker} \int \ee^{\frac{-1}{\kk+n+1}\alpha_i}(z) \ dz  \right) \cap \text{im}\left( \id_{\VOA{H}_\alpha}\otimes \ \varphi \right)\\
    &\cong \left( \bigcap_{i=1}^n \text{ ker} \int \ee^{\frac{-1}{\kk+n+1}\tilde{\alpha_i}}(z) \ dz  \right) \cap \left( \text{ ker} \int \ee^{\frac{1}{2}c + \frac{1}{2}d}(z) \ dz \right),
\end{aligned}
\end{equation}
where $\tilde{\alpha}_1 = \alpha_1 - (\kk+n+1)c$ and $\tilde{\alpha}_i = \alpha_i$ otherwise. As the conformal structure on $\Pi$ defined by $\frac{1}{2}\no{c(z)d(z)}$ gives $\ee^{c}(z)$ conformal dimension zero and $\ee^{-1/(\kk+n+1)\alpha_1}(z)$ has conformal dimension one with respect to $T(z)$, $\ee^{-1/(\kk+n+1)\tilde{\alpha_1}}(z)$ is a screening field on $\VOA{H}_\alpha \otimes \Pi$.

The embedding $\uSubkk \hookrightarrow \VOA{H}_\alpha \otimes \Pi$ defined by \eqref{eq: subreglatticeemb} was first considered by Creutzig, Genra and Nakatsuka where it is used to analyse the coset of $\uSubkk$ by the Heisenberg vertex subalgebra generated by $J(z)$ \cite{CreGen21}. 

The first $n$ screening operators in \eqref{eq: subreglatticeemb} superficially appear to be the screening operators for $\uRegkk$. The difference between these sets of operators is that the $\tilde{\alpha_i}$ screening operators act on both $\VOA{H}_\alpha$ and $\Pi$ non-trivially while the $\alpha_i$ screening operators only act on $\VOA{H}_\alpha$. So the corresponding kernels of vertex operators need not agree. 

To decouple the $\tilde{\alpha}_i$ fields from the rest of $\VOA{H}_\alpha \otimes \Pi$, let $\VOA{H}_{\tilde{\alpha}} \subset \VOA{H}_\alpha \otimes \Pi$ be the vertex subalgebra generated by $\tilde{\alpha}_1(z), \dots, \tilde{\alpha}_n(z)$. It is easy to see that $\VOA{H}_{\tilde{\alpha}} \cong \VOA{H}_{\alpha}$. Let $\tilde{\Pi}$ be the vertex subalgebra of $\VOA{H}_\alpha \otimes \Pi$ generated by
\begin{equation} \label{eq:changeofbasis}
\tilde{c}(z) = c(z), \quad \ee^{m\tilde{c}}(z) = \ee^{mc}(z), \quad \tilde{d}(z) = d(z) - \frac{n}{n+1}(\kk+n+1) c(z) + 2\omega_1(z),
\end{equation}
where $\omega_1(z) = \frac{1}{n+1}\sum_{i=1}^n (n-i+1) \alpha_i(z)$. That is,  $\omega_1(z)$ is the field associated to the first fundamental coweight of $\slnpone$ which is simply laced. As before it is easy to show that $\tilde{\Pi}\cong \Pi$. 

A direct computation shows that the operator product expansion $A(z) B(w)$ of any fields $A(z) \in \VOA{H}_{\tilde{\alpha}}$ and $B(z) \in \tilde{\Pi}$ is nonsingular. Moreover, the expressions in \eqref{eq:changeofbasis} along with those for $\tilde{\alpha}_i(z)$ can be inverted to express the strong generators of  $\VOA{H}_\alpha \otimes \Pi$ in terms of linear combinations of fields in $\VOA{H}_{\tilde{\alpha}}$ and $\tilde{\Pi}$. Therefore $\VOA{H}_{\tilde{\alpha}} \otimes \tilde{\Pi} = \VOA{H}_\alpha \otimes \Pi$. 

By performing this `change of basis', we see that the $\tilde{\alpha_i}$ screening operators are the screening operators for $\uRegkk$ with respect to the Heisenberg vertex algebra $\VOA{H}_{\tilde{\alpha}}$. This leads neatly to the following result.

\begin{theorem} \label{thm:invemb}
Let $\kk$ be generic. There exists an embedding $\uSubkk \hookrightarrow \uRegkk \otimes \Pi$ with
\begin{equation} \label{eq:finalscreening}
    \uSubkk \cong \text{ \emph{ker}} \int  \ee^{a - \omega_1}(z)  \ dz,
\end{equation}
where $\ee^{-\omega_1}(z)$ acts on fields in $\uRegkk$ by way of the strong generators given in the Miura basis \eqref{eq: miurabasis}.
\end{theorem}
\begin{proof}
By \eqref{eq: subreglatticeemb}, the fields of $\uSubkk$ must be of the form $F(z) = \sum_m A_m(z) \otimes B_m(z)$ for some fields $A_m(z) \in \VOA{H}_\alpha$ and $B_m(z) \in \Pi$. By the above discussion, we can also write $F(z) = \sum_m \tilde{A}_m(z) \tilde{B}_m(z)$ for some fields $\tilde{A}_m (z) \in \VOA{H}_{\tilde{\alpha}}$, $\tilde{B}_m(z) \in \tilde{\Pi}$ and inserting tensor products when necessary. 

For convenience, suppose that the fields $\tilde{B}_m(z)$ are linearly independent. If they are not we can simply redefine the fields $\tilde{A}_m(z)$ and reduce the range of $m$ such that this is the case. Since $\tilde{\alpha}_i(z) \tilde{B}_m(w) \sim 0$ for all $i \in \{1,\dots, n\}$ and $m$, $F(z)$ satisfying
\begin{equation} 
    \int \ee^{\frac{-1}{\kk+n+1}\tilde{\alpha_i}}(z) F(w) \ dz = 0
\end{equation}
for all $i \in \{1,\dots, n\}$ is equivalent to
\begin{equation} \label{eq:findingreg}
    \int \ee^{\frac{-1}{\kk+n+1}\tilde{\alpha_i}}(z) \tilde{A}_m(w) \ dz = 0
\end{equation}
for all $i \in \{1,\dots, n\}$ and $m$ by the linear-independence of the fields $\tilde{B}_m(z)$. Therefore if $F(z) \in \uSubkk$,
\begin{equation}
     \tilde{A}_m(w) \in \bigcap_{i=1}^n \text{ker} \int  \ee^{\frac{-1}{\kk+n+1}\tilde{\alpha_i}}(z)  \ dz  \quad \text{for all } m.
\end{equation}
As $\tilde{A}_m(z) \in \VOA{H}_{\tilde{\alpha}}$, this means that $\tilde{A}_m(z) \in \uRegkk$ by \eqref{eq:ffReg}. More precisely, $\tilde{A}_m(z)$ is a normally ordered product of the fields $\tilde{\alpha}_i(z)$ and their derivatives. Imposing \eqref{eq:findingreg} for $i = 1, \dots, n$ on $\tilde{A}_m(z)$ constrains this expansion to be a normally ordered expansion of the fields \eqref{eq: miurabasis} (replacing $\alpha_i$ with $\tilde{\alpha}_i$) and their derivatives, which strongly generate a vertex operator algebra isomorphic to $\uRegkk$. 

Hence we may treat the fields $\tilde{A}_m(z)$ as fields in $\uRegkk$. The fields $\tilde{B}_m(z) \in \tilde{\Pi}$ are unaffected by the $\tilde{\alpha}_i$ screening operators and are therefore unconstrained up to this point. The screening operator from FMS bosonisation \eqref{eq:FMSscreening} present in \eqref{eq: subreglatticeemb} therefore dictates how the fields $\tilde{A}_m(z) \in \uRegkk$ and $\tilde{B}_m(z) \in \tilde{\Pi}$ are combined to form a field in $\uSubkk$. 

The FMS bosonisation screening field $\ee^{\frac{1}{2}c + \frac{1}{2}d}(z)$ can be written in terms of the tilded fields using \eqref{eq:changeofbasis} and the definition of $\VOA{H}_{\tilde{\alpha}}$, and its exponent becomes
\begin{equation}
    \frac{1}{2}c(z) + \frac{1}{2}d(z) =  \left( -\frac{\ell_n(\kk)}{2} \tilde{c}(z) + \frac{1}{2}\tilde{d}(z)  \right)  - \frac{1}{n+1}\sum_{i=1}^n (n-i+1)\tilde{\alpha}_i(z).
\end{equation}
The corresponding map $\uSubkk \rightarrow \uRegkk \otimes \tilde{\Pi}$ is an embedding since its image is equal to the image of the embedding $\uSubkk \hookrightarrow \VOA{H}_\alpha \otimes \Pi$ defined by \eqref{eq: subreglatticeemb}, treating $\uRegkk$ as a vertex subalgebra of $\VOA{H}_{\tilde{\alpha}} \subset \VOA{H}_\alpha \otimes \Pi$. The desired result and screening field description \eqref{eq:finalscreening} follow from the isomorphisms $\VOA{H}_{\tilde{\alpha}} \cong \VOA{H}_\alpha$ and $\tilde{\Pi} \cong \Pi$. 
\end{proof}

\subsection{Explicit expressions} \label{subsec:explicit}

While Theorem \ref{thm:invemb} proves the existence of an embedding $\uSubkk \hookrightarrow \uRegkk \otimes \Pi$ at generic level and gives an associated screening operator, one might also want expressions for the strongly-generating fields \eqref{eq:stronggens} in terms of fields from $\uRegkk$ and $\Pi$. Some fields of $\uSubkk$ can be readily extracted using the screening operator \eqref{eq:finalscreening}. For example, since $ \inner{a}{b} = 0$
\begin{equation}
    b(z) \in \text{ ker} \int  \ee^{a - \omega_1}(z)  \ dz.
\end{equation}
Call the corresponding $\uSubkk$ field $J(z)$. Similarly, the field $\ee^{c}(z)$ has nonsingular operator product expansion with the screening field so we can treat it as a field $G^+(z)$ in  $\uSubkk$. The reason for these assignments is that they reproduce the relevant operator product expansions in \eqref{eq:subopes}. Similarly, the energy-momentum field $L(z)$ of $\uSubkk$ can be decomposed according to
\begin{equation}
    L(z) = T(z) + t(z),
\end{equation}
as the operator product expansion of the right-hand-side with the screening operator in \eqref{eq:finalscreening} is a total derivative. In particular, this shows that the embedding $\uSubkk \hookrightarrow \uRegkk \otimes \Pi$ is conformal (with the chosen conformal structure on $\Pi$). 

To find similar expressions for the remaining strong generators of $\uSubkk$, we use explicit expressions for the aforementioned strong generators of $\uSubkk$ \cite{GenStrong20}. These expressions rely on the fact that the subregular W-algebra of type A $\uSubkk$ is isomorphic to the Feigin--Semikhatov algebra $\FS{n+1}$ \cite[Sec. 6.2]{Gen17} and screening operators for the latter are known \cite{FS04}. 

The screening operators of $\uSubkk$ from \cite{FS04} act on a vertex algebra $\mathcal{H}^\kk$ generated by fields $A_1(z), \dots, A_n(z), Q(z), Y(z)$ along with vertex operators $\ee^{m Y}(z)$ where $m \in \ZZ$. To define the operator product expansions of these fields, let $\mathbb{V} = \CC A_n \oplus \dots \oplus \CC A_1 \oplus \CC Q \oplus \CC Y$ and define a symmetric bilinear form $( \cdot, \cdot)$ on $\mathbb{V}$ by
\begin{equation}
    (A_i, A_j) = A_{i,j}(\kk+n+1), \quad (A_1, Q) = -(\kk+n+1), \quad (Q,Q) = 1, \quad (Q,Y) = 1,
\end{equation}
with all omitted evaluations yielding zero. The operator product expansions of $\mathcal{H}^\kk$ can be written as, for all  $A,B \in \mathbb{V}$,
\begin{equation}
    A(z) B(w) \sim \frac{(A,B)}{(z-w)^2}, \quad \ee^{mY}(z) \ee^{nY}(w) \sim 0, \quad A(z) \ee^{mY}(w) \sim \frac{(A,mY)\ee^{mY}(w)}{z-w}..   
\end{equation}
In addition to $\ee^{mY}(z)$, we also have vertex operators $\ee^{C}(z)$ where $C \in \{A_n, \dots, A_1, Q\}$ with operator product expansions
\begin{equation}
    A(z)\ee^{C}(w) \sim \frac{(A,C)\ee^{C}(w)}{z-w}, \quad e^{A_i}(z) \ee^{mY}(w) \sim 0, \quad e^{Q}(z) \ee^{mY}(w) \sim (z-w)^m \ee^{mY+Q}(w) + \dots.
\end{equation}
Proposition 2.1 in \cite{GenStrong20} defines an embedding $\uSubkk \hookrightarrow \mathcal{H}^\kk$ at generic level with
\begin{equation} \label{eq:genrascreening}
    \uSubkk \cong \left( \text{ ker} \int \ee^{Q}(z) \ dz \right) \cap \left( \bigcap_{i=1}^n \int \ee^{A_i}(z) \ dz   \right).
\end{equation}
Expressions for the strong generators $L$, $G^+$, $J$, $U_3$, \dots, $U_n$, $G^-$ in terms of the fields in $\mathcal{H}^\kk$ can then be written down explicitly for all $\kk$ \cite{GenStrong20}. So we have two different free-field realisations of $\uSubkk$, one of which has been analysed further to obtain expressions for strong generators of $\uSubkk$. In order to use the $\mathcal{H}^\kk$ expressions in our present setting of $\VOA{H}_\alpha \otimes \Pi$, we must understand how $\mathcal{H}^\kk$ and \eqref{eq:genrascreening} relate to $\VOA{H}_\alpha \otimes \Pi$ and \eqref{eq: subreglatticeemb}.

\begin{proposition} \label{prop:freefieldiso}
Define $\psi: \mathcal{H}^\kk \rightarrow \VOA{H}_{\alpha} \otimes \Pi$ to be the vertex algebra map defined by
\begin{equation}
    A_i(z) \mapsto \alpha_i(z), \quad Y(z) \mapsto c(z), \quad \ee^{mY}(z) \mapsto \ee^{mc}(z), \quad Q(z) \mapsto a(z) - \omega_1(z). 
\end{equation}
Then $\psi$ is a vertex algebra isomorphism.
\end{proposition}

\noindent As this is an isomorphism, the kernel of the screening operator $\int \ee^{Q}(z) \ dz$ is equal to $\psi^{-1}$ applied to the kernel of the screening operator in \eqref{eq:finalscreening}. The kernel of the screening operator $\int \ee^{A_i(z)} \ dz$ is equal to $\psi^{-1}$ applied to the kernel of $\int \ee^{\alpha_i}(z) \ dz$. By Feigin--Frenkel duality for the Virasoro vertex algebra (see \cite[Ch.~15]{FBZ04}),
\begin{equation}
    \text{ ker} \int \ee^{\alpha_i}(z) \ dz = \text{ ker} \int \ee^{\frac{-1}{\kk+n+1}\alpha_i}(z) \ dz
\end{equation}
for $i = 1,\dots, n$. So the isomorphism $\psi$ maps the intersection of kernels in \eqref{eq:genrascreening} to that in \eqref{eq: subreglatticeemb}. In other words, $\psi$ intertwines the action of the screening operators on $\mathcal{H}^\kk$ and  $\VOA{H}_{\alpha} \otimes \Pi$ that define embeddings of $\uSubkk$. 

To summarise what we have shown thus far, applying $\psi$ to the strong generators of $\uSubkk$ in terms of the fields of $\mathcal{H}^\kk$ presented in \cite{GenStrong20} gives strong generators of $\uSubkk$ in terms of fields of $\VOA{H}_\alpha \otimes \Pi$. The latter set of strong generators also belong to (and therefore strongly generate) the intersection of kernels in \eqref{eq: subreglatticeemb}, and by Theorem \ref{thm:invemb} must consist of fields of $\uRegkk \otimes \Pi$ only, treating $\uRegkk$ as a subalgebra of $\VOA{H}_\alpha$ via the Miura transformation. 

To write down the generators obtained by applying $\psi$ to those in \cite{GenStrong20}, let $\wun_R$ and $\wun_\Pi$ be the vacuum states of $\uRegkk$ and $\Pi$ respectively. Then the vacuum of $\uRegkk \otimes \Pi$ is $\wun = \wun_R \otimes \wun_\Pi$. We will frequently omit tensor product symbols in what follows when it is clear which modes act on which component of $\VOA{H}_\alpha \otimes \Pi$. Define operators $\rho_0, \rho_1 ,\dots, \rho_{n+1}$ on $\VOA{H}_\alpha \otimes \Pi$ by
\begin{equation}
    \rho_0 = (\kk+n) (\partial + c_{-1}), \quad \rho_i = (\kk+n)\partial + b_{-1}  + \frac{\kk+n+1}{n+1}c_{-1} - (\varepsilon_i)_{-1}.
\end{equation}
As $T + t$ is a conformal vector, we may write $\partial = T_{-1}+t_{-1}$. Observe that $\rho_i \ee^{-c} = a_{-1}\ee^{-c} + \left( (\kk+n)T_{-1} -(\varepsilon_i)_{-1} \right)\ee^{-c}$.

\begin{proposition} \label{prop:stronggensubreg}
Let $\kk \neq -n-1$. Define the fields $L(z)$, $G^+(z)$, $J(z)$, $U_3(z)$, \dots, $U_n(z)$, $G^-(z) \in \VOA{H}_\alpha \otimes \Pi$ by
\begin{equation} \label{eq:stronggensubreg}
\begin{gathered}
    L = T + t, \quad J = b, \quad G^+ = \ee^{c}, \\
    G^- = -E_{n+1}(\rho_1, \dots, \rho_{n+1})\ee^{-c} \\
    U_i = \sum_{j=0}^i (-1)^{i+j} \left(\prod_{m=1}^j \frac{m(\kk+n) + 1}{m(\kk+n)}  \right)E_{i-j}(\rho_1, \dots, \rho_{n+1})\rho_0^{j} \wun. 
\end{gathered}
\end{equation}
The associated fields strongly-generate $\uSubkk$ and satisfy the operator product expansions \eqref{eq:subopes}.
\end{proposition}
\begin{proof}
These fields are $\psi$ applied to the strong generators described in Remark 3.14 in \cite[arXiv]{GenStrong20} except for $L(z)$. The relationship between $L(z)$ and the field $U_2(z)$, which is defined using the above formula for $U_i(z)$ is a straightforward calculation and is of the form
\begin{equation}
    L(z) = \frac{1}{\kk+n+1}\left(-U_2(z) + a_1 \partial J(z) + a_2 \no{J(z)J(z)}\right),
\end{equation}
with $a_1, a_2 \in \RR[\kk]$ whose precise form is not important. Since the set $\set{G^+(z)$, $J(z)$, $U_2(z)$, $U_3(z)$, \dots, $U_n(z)$, $G^-(z)}$ strongly-generates $\uSubkk$, the same set with $U_2(z)$ replaced with $L(z)$ also strongly-generates $\uSubkk$.
\end{proof}

\noindent The apparent singularity at $\kk = -n$ in the formula for $U_i$ is taken care of by the factor of $(\kk+n)$ in $\rho_0$. To see how the fields above decompose in terms of $\uRegkk$ fields, notice that the only place where fields of $\VOA{H}_\alpha$ appear (outside of $T$ in $L$) is in the symmetric polynomials involving the operators $\rho_i$: Since the vacuum is translation invariant, $\rho_0 \wun = (\kk+n)c$ and successive applications of $\rho_0$ do not introduce any additional modes/states from $\VOA{H}_\alpha$ as $T_{-1}$ commutes with all modes in $\Pi$ and $T_{-1} \wun = (T_{-1}\wun_R) \otimes \wun_\Pi = 0$. Therefore all that remains is to determine which $\uRegkk$ modes appear in $E_m(\rho_1, \dots, \rho_{n+1})$.

\begin{lemma} \label{lem:sympolys}
Let $ m \in \{1, \dots, n+1 \}$.
\begin{equation} \label{eq:sympolys}
    E_m(\rho_1, \dots, \rho_{n+1}) = \sum_{j=0}^m \binom{n+1-j}{m-j}E_j(\sigma_1, \dots, \sigma_{n+1}) \left( (\kk+n)t_{-1} + b_{-1}  + \frac{\kk+n+1}{n+1}c_{-1}  \right)^{m-j},
\end{equation}
where $\sigma_i = (\kk+n)T_{-1}-(\varepsilon_i)_{-1}$.
\end{lemma}
\begin{proof}
This is Proposition 12.4.4 in \cite{Molev18} with $N = n+1$,
\begin{equation}
    \tau + \mu_{i}[-1] = \sigma_{i}, \quad u = (\kk+n)t_{-1} + b_{-1}  + \frac{\kk+n+1}{n+1}c_{-1}
\end{equation}
and noting that under this identification, $\rho_i = u+\tau + \mu_{i}[-1]$.
\end{proof}

\noindent The above lemma shows that the only $\VOA{H}_\alpha$ fields that appear in the free-field realisation of $\uSubkk$ in Proposition \ref{prop:stronggensubreg} are those given by \eqref{eq:freefieldforregfields}, i.e. fields of $\uRegkk$ as expected. The `$\uRegkk$ content' of the strong generators $L(z)$, $G^+(z)$, $J(z)$ is clear from the formulae given earlier. For the fields $U_3(z)$, \dots, $U_n(z)$, $G^-(z)$, we have the following structural result:

\begin{theorem} \label{thm:regcontent}
Let $ i \in \{ 3, \dots, n\}$. There exist $i$ fields $\picorr{i}{j}(z)$ in the vertex subalgebra of $\Pi$ generated by $c(z)$ and $d(z)$ such that
\begin{equation} \label{eq:regcontentcasimirs}
    U_i(z) = (-1)^{i+1} W_i(z) + \sum_{j=0}^{i-1} W_j(z) \otimes \picorr{i}{j}(z). 
\end{equation}
The field $G^-(z)$ can be written as
\begin{equation}
    G^-(z) = W_{n+1}(z) \otimes \ee^{-c}(z) +   \sum_{j=0}^{n} W_j(z) \otimes \left( \picorr{-}{j}_{(-1)} \ee^{-c}(z) \right),
\end{equation}
where $\picorr{-}{j}(z) \in \Pi$ is given by
\begin{equation}
    \picorr{-}{j}(z) = \big( (\kk + n) t_{-1} + a_{-1} \big)^{n+1-j}\wun_\Pi(z).
\end{equation}
\end{theorem}
\begin{proof}
Starting with $G^-(z)$, substituting \eqref{eq:sympolys} into the expression for $G^-$ in Proposition \ref{prop:stronggensubreg} gives
\begin{equation}
\begin{aligned}
    G^- &= -\sum_{j=0}^{n+1} \left( E_j(\sigma_1, \dots, \sigma_{n+1}) \wun_R  \right) \otimes \left( \left( (\kk+n)t_{-1} + b_{-1}  + \frac{\kk+n+1}{n+1}c_{-1}  \right)^{n+1-j}\ee^{-c} \right) \\
    &=  \sum_{j=0}^{n+1} W_j \otimes \left( \tilde{\pi}^{-,j}_{(-1)}\ee^{-c} \right)
\end{aligned}
\end{equation}
for some field $\tilde{\pi}^{-,j} (z) \in \Pi$. That $\tilde{\pi}^{-,j} = \picorr{-}{j}$ follows from the action $t_{-1} \ee^{-c} = -c_{-1}\ee^{-c}$ and the fact that negative modes in $\Pi$ commute. The decomposition of $U_i(z)$ in terms of $W_j(z)$ for $j \le i$ is similarly obtained by substituting \eqref{eq:sympolys} into the expression for $U_i$ given in Proposition \ref{prop:stronggensubreg}. 
\end{proof}

\noindent A consequence of these formulae is that the embedding $\uSubkk \hookrightarrow \uRegkk \otimes \Pi$ exists for all non-critical $\kk$. This is because the strong generators $\set{G^+,J,L,U_3,\dots,U_n,G^-}$ and $\set{T,W_3,\dots,W_{n+1}}$ of $\uSubkk$ and $\uRegkk$ respectively exist for all non-critical $\kk$, and the decompositions of the former in terms of the latter and states in $\Pi$ are well-defined for all $\kk \neq -n-1$. Injectivity of the corresponding vertex operator algebra homomorphism $\uSubkk \rightarrow \uRegkk \otimes \Pi$ follows from the injectivity of the free-field realisations $\uSubkk \hookrightarrow \VOA{H}_\alpha \otimes \Pi$ and $\uRegkk \hookrightarrow \VOA{H}_{\tilde{\alpha}}$.

When $\kk$ is generic, \cref{prop:freefieldiso} shows that the explicit embedding described in this section is indeed the one specified by the screening operator in \eqref{eq:finalscreening}.

The critical level embedding of vertex algebras $\uSub{-n-1}{n+1} \hookrightarrow \uReg{-n-1}{n+1} \otimes \Pi$, as described in \cite{GenStrong20}, is obtained using the same formulae as in the non-critical case except for multiplying $L = T+t$ by $(\kk+n+1)$.

An important feature of these decompositions is that the Miura basis $\uRegkk$ fields appear undisturbed in the sense that their derivatives and normally ordered products do not appear. Moreover, all Miura basis fields of $\uRegkk$ appear somewhere in these expansions. These properties will allow us to prove the almost-simplicity of a large class of $\uSubkk$-modules that we construct using the inverse reduction embedding in an analogous way to that used in Section 5 of \cite{AdaRea20}. 

We conclude this section with a few examples. The convention for normally ordered products used here is right-to-left. For example, $\no{A(z)B(z)C(z)} = \no{A(z)\big(\no{B(z)C(z)}\big)}$.

\begin{example}[$n=1$]
Denote the generators of $\uSub{\kk}{2} = \uaffvoa{\kk}{\sltwo}$ by $h(z) = 2 J(z)$, $e(z)=G^+(z)$ and $f(z) = G^-(z)$. The regular W-algebra $\uReg{\kk}{2}$ is isomorphic to the Virasoro vertex operator algebra $\VOA{Vir}^\kk$ generated by a single energy-momentum field $T(z) = \frac{1}{\kk+2}W_2(z)$ with central charge $\cc^2_\kk$. The embedding $\uaffvoa{\kk}{\sltwo} \hookrightarrow \VOA{Vir}^\kk \otimes \Pi$ for $\kk \neq -2$ given by Theorem \ref{thm:invemb} is
\begin{equation}
    h(z) \mapsto 2 b(z), \quad e(z) \mapsto \ee^{c}(z), \quad f(z) \mapsto (\kk+2)T(z)\ee^{-c}(z) - \no{ (\kk+1)\partial a(z) + a(z)^2\big)\ee^{-c}(z)}
\end{equation}
where $a(z) = -\frac{\kk}{4}c(z) + \frac{1}{2}d(z)$ and $b(z) = \frac{\kk}{4}c(z) + \frac{1}{2}d(z)$. This is the embedding described in Section 3 of \cite{AdaRea17}.
\end{example}

\begin{example}[$n=2$]
$\uSub{\kk}{3} \cong \ubpvoak$ has strong generators denoted by $J(z)$, $L(z)$, $G^+(z)$ and $G^-(z)$. The regular W-algebra $\uReg{\kk}{3}$ (also known as the Zamolodchikov algebra) is generated by fields $T(z)$, $W(z)$ of conformal dimension $2$ and $3$ respectively. The operator product expansions for $\uWvoak$ are well-known and can be found in Appendix \ref{subsec:A2}. The embedding $\ubpvoak \hookrightarrow \uWvoak \otimes \Pi$ for $\kk \neq -3$ given by Theorem \ref{thm:invemb} is
\begin{equation}
\begin{gathered}
    J(z) \mapsto b(z), \quad L(z) \mapsto T(z) + t(z), \quad G^+(z) \mapsto \ee^{c}(z), \\
    G^-(z) \mapsto W_3(z)\ee^{-c}(z) + W_2(z)\no{a(z)\ee^{-c}(z)}-\no{\big((\kk+2)^2\partial^2 a(z)+3(\kk+2)a(z)\partial a(z)+ a(z)^3 \big)\ee^{-c}(z)}
\end{gathered}
\end{equation}
where 
\begin{equation}
    a(z) = -\frac{2\kk+3}{6}c(z) + \frac{1}{2}d(z), \quad 
    b(z) = \frac{2\kk+3}{6}c(z) + \frac{1}{2}d(z), \quad 
    t(z) = \frac{1}{2}\no{c(z)d(z)} + \frac{2\kk+3}{3}\partial c(z) -\frac{1}{2}\partial d(z).
\end{equation}
This is the embedding described in Section 3.3 of \cite{AdaRea20} as the Miura basis fields of $\uWvoak$ are related to $T(z)$ and $W(z)$ by
\begin{equation}
    T(z) = \frac{1}{\kk+3}W_2(z), \quad W(z) = W_3(z) -\frac{1}{2}(\kk+2)\partial W_2(z).
\end{equation}
\end{example}

\begin{example}[$n=3$]
$\uSub{\kk}{4}$ is our first new example. Once again we denote its strong generators by $J(z)$, $L(z)$, $U_3(z)$ $G^+(z)$ and $G^-(z)$. Denote the Miura basis fields for the corresponding regular W-algebra $\uReg{\kk}{4}$ by $T(z)$, $W_3(z)$, $W_4(z)$ as usual. The relevant basis for the Heisenberg fields in $\Pi$ consists of the fields
\begin{equation}
    a(z) = -\frac{3\kk+8}{8} c(z) +  \frac{1}{2}d(z), \quad b(z) = \frac{3\kk+8}{8} c(z) +  \frac{1}{2}d(z)
\end{equation}
while the conformal structure is furnished by 
\begin{equation}
    t(z) = \frac{1}{2}\no{c(z)d(z)} + \frac{3(3\kk+8)}{8}\partial c(z) -\frac{1}{2}\partial d(z).
\end{equation}
The embedding $\uSub{\kk}{4} \hookrightarrow \uReg{\kk}{4} \otimes \Pi$ for $\kk \neq -4$ given by Theorem \ref{thm:invemb} is
\begin{equation*}
    J(z) \mapsto b(z), \quad  L(z) \mapsto T(z) + t(z), \quad G^+(z)(z) = \ee^{c}(z),
\end{equation*}
\begin{equation}
\begin{aligned}
    G^-(z) \mapsto& \ W_4(z) \otimes \ee^{-c}(z) + W_3(z) \otimes \no{a(z)\ee^{-c}(z)} + W_2(z) \otimes \no{\left((\kk+3) \partial a(z) + a(z)^2 \right)\ee^{-c}(z)} \\ 
            &-\no{\left(a(z)^4 + (\kk+3)^3 \partial^3 a(z)+ 3(\kk+3) \partial a(z)^2 + 4(\kk+3)^2 a(z) \partial^2 a(z) + 6 a(z)^2 \partial a(z) \right)\ee^{-c}(z)},\\
\end{aligned}
\end{equation}
\begin{equation*}
\begin{aligned}
U_3(z) \mapsto& \ W_3(z) + 2 W_2(z) \otimes m(z) -4(\kk+3)^2\partial^2 m(z) - 12(\kk+3) \no{m(z)\partial m(z)} - 4\no{m(z)^3} \\
    & \hspace{-15pt}+ (\kk+4) \left( - W_2(z) \otimes c(z) + 6(\kk+3)^2 \partial^2 c(z) + 6(\kk+3)\no{\partial m(z) c(z)} + 12(\kk+3)\no{m(z) \partial c(z)} + 6\no{m(z)^2 c(z)}  \right)\\
    & \hspace{10pt} - 2(\kk+4)(2\kk+7) \left((\kk+3)(\partial^2 c(z) + \no{c(z)\partial c(z)}) + \no{m(z) \partial c(z)} + \no{m(z) c(z)^2}  \right) \\
    & \hspace{30pt} +\frac{(\kk+4)(2\kk+7)(3\kk+10)}{6} \left(\partial^2 c(z) + 3 \no{c(z) \partial c(z)} + \no{c(z)^3} \right),
\end{aligned}
\end{equation*}
where $m(z) = b(z) + \frac{\kk+4}{4}c(z) \in  \Pi$. Define the field
\begin{equation}
    W(z) = -\frac{1}{\kk+2}U_3(z) - \frac{2(2\kk+5)}{3}\partial^2 J(z) + \frac{4(\kk+4)}{3\kk+8}\no{J(z)L(z)} - 6 \no{J(z)\partial J(z)} + \frac{\kk+4}{2}\partial L(z) - \frac{8(11\kk+32)}{3(3\kk+8)^2}\no{J(z)^3}.
\end{equation}
This is a primary field of conformal dimension 3. The fields $\set{J(z), L(z), W(z), G^+(z), G^-(z)}$ also strongly generate $\uSub{\kk}{4}$. Checking that the operator product expansions \eqref{ope:FS} are correctly reproduced can be done tediously by hand or quickly using a computer using tools such as the Mathematica package \emph{OPEdefs} \cite{OPEdefs}.
\end{example}

\section{Relaxed modules for subregular W-algebras} \label{sec:relaxedmods}
The existence of an embedding $\uSubkk \hookrightarrow \uRegkk \otimes \Pi$ at non-critical level allows us to construct infinite families of $\uSubkk$-modules by taking appropriate tensor products of $\uRegkk$- and $\Pi$-modules. Particularly important amongst modules for subregular W-algebras are relaxed modules. 

These play a central role in computing modular transformations and fusion rules for admissible-level $\saffvoa{\kk}{\sltwo}$ and admissible-level $\sbpvoak$. In both cases, this is because relaxed modules can be realised in terms of modules for the corresponding regular W-algebra and for $\Pi$ using inverse quantum hamiltonian reduction. This is described in detail in Section 7.1 of \cite{AdaRea17} for $n=1$ and Section 3.3 of \cite{AdaRea20} for $n=2$. It is therefore reasonable to anticipate that some of the $\uSubkk$-modules constructed by way of our embedding will play a central role in the representation theory of $\sSubkk$ and the construction of `logarithmic minimal models' with $\sSubkk$ symmetry. 

Much of this section follows the approach taken for $n=2$. As we are interested in extracting information out of this embedding, we assume that $\kk$ is non-critical for the remainder of this paper. We will also frequently identify $\uSubkk$ with its image under the embedding obtained in Proposition \ref{prop:stronggensubreg}.

\subsection{Relaxed modules} \label{subsec:relaxedmod}
Before constructing $\uSubkk$-modules, we recall some general properties and definitions of modules over vertex operator algebras as used in the analysis of relaxed modules for Bershadsky--Polyakov algebras in \cite{AdaRea20}.

\begin{definition}
Let $V$ be a vertex operator algebra, $M$ a $\ZZ_{\ge 0}$-graded $V$-module and denote the top space of $M$ by $M_{\text{top}} = M_0$.
\begin{itemize}
\item $M$ is \emph{top-generated} if $M$ is generated by $M_\text{top}$.
\item $M$ has \emph{only top-submodules} if every nonzero submodule of $M$ has nonzero intersection with $M_\text{top}$.
\item $M$ is \emph{almost-irreducible} if it is top-generated and has only top-submodules.
\end{itemize}
\end{definition}

\noindent By Zhu's theorem \cite{ZhuMod96}, $M_\text{top}$ is a module for the associative algebra $\zhu{V}$. When looking to construct irreducible $\VOA{V}$-modules, it is convenient to consider $\VOA{V}$-modules $M$ whose submodules are all generated by $\zhu{\VOA{V}}$-submodules of $M_{\text{top}}$. This is because checking if such a $\VOA{V}$-module $M$ is irreducible requires only checking if $M_{\text{top}}$ is an irreducible $\zhu{\VOA{V}}$-module:

\begin{proposition}[Proposition 5.2 \cite{AdaRea20}]
If $M$ is almost-irreducible and $M_\text{top}$ is irreducible as a $\zhu{V}$-module then $M$ is irreducible.
\end{proposition}

\noindent By the embedding \eqref{eq:stronggensubreg}, there are two strong generators of $\uSubkk$ living only in $\Pi$: $G^+ \rightarrow \ee^c$ and $J\rightarrow b$. Let $U$ be the vertex subalgebra of $\Pi$ generated by $b$ and $\ee^c$. As a $U$-module, $\halflatticemod{-1}{\lambda}$ is also particularly nice.

\begin{proposition}
$\halflatticemod{-1}{\lambda}$ is almost-irreducible as a $U$-module.
\end{proposition}
\begin{proof}
The proof for all $n$ is identical to that for the $n=2$ case described in Section 5.1 of \cite{AdaRea20}. Adapting it to the case of general $n$ simply requires replacing $j$ and $i$ with $b$ and $a$ respectively and keeping in mind the different conformal structures.
\end{proof}

\noindent Again, under the identification described in \cref{subsec:embedding}, the generators of $\wun_R \otimes U$ are elements of $\uSubkk$. So $\wun_R \otimes U$ is also a vertex subalgebra of $\uSubkk$. 

Given a $\uRegkk$-module $M$, the $\uRegkk\otimes\Pi$-module $\tpmod{M}{r,\lambda} = M \otimes \halflatticemod{r}{\lambda}$, where $r\in \ZZ$ to ensure $\ZZ$-grading, is also a $\uSubkk$-module by restriction. As $\halflatticemod{r}{\lambda} \cong \halflatticemod{r}{\lambda+n}$ as $\Pi$-modules for all $n \in \ZZ$, $\tpmod{M}{r,\lambda} \cong \tpmod{M}{r,\lambda + n}$ as $\uSubkk$-modules. Additionally, $J(z) \in \uSubkk$ is identified with $b(z)$ under the embedding defined by Theorem \ref{thm:invemb} so applying the $\uSubkk$ version of spectral flow to $\tpmod{M}{r,\lambda}$ can be performed purely in terms of the $\Pi$ version of spectral flow:
\begin{equation}
    \sfmod{\ell}{\tpmod{M}{r,\lambda}} 
    = \sfmod{\ell}{M \otimes \halflatticemod{r}{\lambda}} 
    = M \otimes \rho^\ell\left( \halflatticemod{r}{\lambda} \right) = M \otimes \halflatticemod{r+\ell}{\lambda} 
    = \tpmod{M}{r+\ell,\lambda}.
\end{equation}
We therefore interpret the label `$r$' in $\tpmod{M}{r,\lambda}$ as a spectral flow index. Owing again to the simplicity of the expressions for $J(z)$ and $L(z)$ in terms of $\uRegkk$ and $\Pi$ fields, character formulae for $\tpmod{M}{r,\lambda}$ are immediate from their construction. 

\begin{corollary} \label{cor:char}
Suppose that $M$ is a $\uRegkk$-module with $q$-character $\qcharac{M}{q} = \textup{tr}_M\left(q^{T_0 - \Regcck/24}\right)$. Then the $\uSubkk$-module $\tpmod{M}{r,\lambda}$ has character
\begin{equation}
\begin{aligned}
    \charac{\tpmod{M}{r,\lambda}}{z}{q}
    &= \emph{tr}_{\tpmod{M}{r,\lambda}} \left(z^{J_0} q^{L_0 -\frac{\Subcck}{24}} \right) \\
    &= \qcharac{M}{q}\charac{\halflatticemod{r}{\lambda}}{z}{q}\\
    &= \qcharac{M}{q} z^{(r+1)\ell_n(\kk)}q^{(r+1)(r+2-n)\frac{\ell_n(\kk)}{2}}\charac{\halflatticemod{-1}{\lambda}}{z q^{r+1}}{q}
\end{aligned}
\end{equation}
where $\charac{\halflatticemod{-1}{\lambda}}{z}{q}$ is given by \eqref{eq:minusonelatticemodcharac}. 
\end{corollary}

\noindent It is useful to know what properties of $M$ are inherited by $\tpmod{M}{-1,\lambda}$  (recall that $\halflatticemod{r}{\lambda}$ is positive-energy with respect to $t(z)$ only when $r = -1$). For example, only having top-submodules and being top-generated. Fortunately, $\halflatticemod{-1}{\lambda}$ being an almost-irreducible $U$-module is strong enough to require fairly mild constraints on the $M$ for which such properties are inherited by $\tpmod{M}{-1,\lambda}$. As in the $n=2$ case, it is convenient to introduce $U_{n+1} = (-1)^n G^-_{(-1)}G^+ \in \uSubkk$. This field can be expanded as
\begin{equation}
    U_{n+1}(z) = (-1)^n W_{n+1}(z) + \sum_{j=0}^n W_j(z) \otimes \picorr{n+1}{j}(z) 
\end{equation}
for some fields $\picorr{n+1}{j}(z)$ in the vertex subalgebra of $\Pi$ generated by $c(z)$ and $d(z)$.  The following theorems are generalisations of Theorems 5.9 and 5.10 in \cite{AdaRea20}. The main difference when $n>2$ is the existence of strong-generating fields $U_i(z)$. This would be an issue if not for the structure of the decompositions of such fields in terms of fields in $\uRegkk$ (i.e. the lack of derivatives or normally ordered products).

\begin{theorem}
If $M$ is a weight $\uRegkk$-module that has only top-submodules, then the weight $\uSubkk$-module $\tpmod{M}{-1,\lambda}$ has only top-submodules  for all $\lambda \in \CC$.
\end{theorem}
\begin{proof}
The proof used here follows the same approach used for $n=2$ in Theorem 5.9 of \cite{AdaRea20}. Assume that $N$ is a nonzero $\uSubkk$-submodule of $\tpmod{M}{-1,\lambda}$ and let $w \in N$ be a weight vector. As $\halflatticemod{-1}{\lambda}$ has only top-submodules as a $U$-module, $w$ can be sent to a nonzero element of $M \otimes {\halflatticemod{-1}{\lambda}}_\text{top}$ under the action of modes from $\wun_R \otimes U \subset \uSubkk$. The result is therefore an element $w_0 = u_0 \otimes v_\text{top} \in N$ where $u_0 \in M$ and $v_\text{top} \in {\halflatticemod{-1}{\lambda}}_\text{top}$. 

If we can show that applying suitable modes from $\uSubkk$ to $w_0$ results in an element of $\tpmod{M}{-1,\lambda}_\text{top} = M_\text{top} \otimes {\halflatticemod{-1}{\lambda}}_\text{top}$ then we are done. We will do this recursively by defining $w_1, \dots, w_k \in N$ with $w_p = u_p \otimes v_\text{top}$ such that the conformal weight strictly decreases at each step and $w_k \in \tpmod{M}{-1,\lambda}_\text{top}$ for some $k \in \ZZ_{\ge 0}$. In other words, we recursively move up the submodule until we reach the top space and do so in such a way that $v_\text{top}$ is untouched at each step. 

To define this recursion, suppose $w_p = u_p \otimes v_\text{top} \in N$. If $(W_j)_m u_p = 0$ for all $j$ and $m>0$, then $u_p$ generates a highest-weight submodule of $M$. Since $M$ only has top-submodules, it follows that $u_p \in M_\text{top}$ and therefore that $w_p \in \tpmod{M}{-1,\lambda}_\text{top}$. Otherwise, let $i \in \{2,...,n+1\}$ be the unique integer for which
\begin{equation}
    (W_{j})_\ell u_p = 0 \quad \text{for all }\ell \in \ZZ_{>0} \text{ and } j<i,
\end{equation} 
and there exists $m' \in \ZZ_{>0}$ such that $(W_i)_{m'} u_p \neq 0$. Amongst $m'$ satisfying this condition, let $m$ be maximal for definiteness. Define $w_{p+1} = (U_i)_m w_p$. It is clear that $w_{p+1} \in N$ and has conformal weight strictly smaller than $w_p$. Moreover,
\begin{equation}
\begin{aligned}
    w_{p+1}  &= \left((-1)^{i+1} W_i(z)  + \sum_{j=0}^{i-1} W_j(z) \otimes \picorr{i}{j} (z)   \right)_m  u_p \otimes v_\text{top} \\
    &= (-1)^{i+1}(W_i)_m u_p \otimes v_\text{top} + \sum_{j=0}^{i-1} \sum_{r=0}^\infty (W_j)_{m+r} u_p \otimes \picorr{i}{j}_{-r} v_\text{top}\\
    &= (-1)^{i+1}(W_i)_m u_p \otimes v_\text{top} \neq 0.
\end{aligned}
\end{equation}
As the conformal weight of $u_p$ decreases at each iteration and $M$ is positive-energy, applying this procedure sufficiently many times will yield a nonzero $u_k \in M_\text{top}$. That is, $w_k \in  \tpmod{M}{-1,\lambda}_\text{top} \cap N$ and therefore $\tpmod{M}{-1,\lambda}$ has only top-submodules.
\end{proof}

\begin{theorem} \label{thm:topgenerated}
If $M$ is a top-generated weight $\uRegkk$-module, then $\tpmod{M}{-1,\lambda}$ is a top-generated weight $\uSubkk$-module for all $\lambda \in \CC$.
\end{theorem}
\begin{proof}
The proof used here follows the same approach used for $n=2$ in Theorem 5.10 of \cite{AdaRea20}. Let $N$ be the submodule of $\tpmod{M}{-1,\lambda}$ generated by the top space $\tpmod{M}{-1,\lambda}_\text{top} = M_\text{top} \otimes {\halflatticemod{-1}{\lambda}}_\text{top}$. We begin by showing that $M \otimes {\halflatticemod{-1}{\lambda}}_\text{top} \subset N$. Assuming this, it then follows that $\tpmod{M}{-1,\lambda} \subset N$. This is because ${\halflatticemod{-1}{\lambda}}$ being top generated as a $U$-module means that any $u \otimes v \in \tpmod{M}{-1,\lambda}$ can be written as a collection of modes from $\wun_R \otimes U \subset \uSubkk$ acting on $u \otimes v_\text{top}$ for some $v_\text{top} \in {\halflatticemod{-1}{\lambda}}_\text{top}$. 

With this in mind, let $u \otimes v_\text{top} \in M \otimes {\halflatticemod{-1}{\lambda}}_\text{top}$. If $u \in M_\text{top}$, then it is clear that $u \otimes v_\text{top} \in N$. Otherwise, as $M$ is top-generated, $u$ is obtained from some $u_\text{top} \in M_\text{top}$ by the application of modes of the fields $W_2(z), \dots, W_{n+1}(z)$. 

It therefore suffices to show that if $u \otimes v_\text{top} \in N$ for all $v_\text{top} \in {\halflatticemod{-1}{\lambda}}_\text{top}$, then so is $(W_i)_{-m} u \otimes v_\text{top}$ for all $i = 2, \dots, n+1$ and $m>0$. Starting with $i=2$, 
\begin{equation}
    (U_2)_{-m} \left(u \otimes v_\text{top}\right) =-(W_2)_{-m} u \otimes v_\text{top} -  u \otimes \picorr{2}{0}_{-m} v_\text{top} \in N.
\end{equation}
As $\halflatticemod{-1}{\lambda}$ is top-generated as a $U$-module, $\picorr{2}{0}_{-m} v_\text{top}$ can be obtained from some $v'_\text{top}$ by the action of modes from $U$. So the right most term above can be written as modes from $\wun_R \otimes U \subset \uSubkk$ acting on $u \otimes v'_\text{top} \in N$ and is therefore also in $N$. Hence $(W_2)_{-m} u \otimes v_\text{top} \in N$ for all $m>0$. For $i=3$, 
\begin{equation}
\begin{aligned}
    (U_3)_{-m} \left(u \otimes v_\text{top}\right) &= (W_3 + W_2 \otimes \picorr{3}{2} - \picorr{3}{0})_{-m}     u \otimes v_\text{top}\\
    &= (W_3)_{-m} u \otimes v_\text{top}  + \sum_{r=0}^{\infty}(W_2)_{-m+r} u \otimes \picorr{3}{2}_{-r} v_\text{top}  - u \otimes \picorr{3}{0}_{-m} v_\text{top}    \in N.
\end{aligned}
\end{equation}
An identical argument as in the $i=2$ case shows that $u \otimes \picorr{3}{0}_{-m} v_\text{top}    \in N$. That $(W_2)_{-m+r} u \otimes \picorr{3}{2}_{-r} v_\text{top} \in N$ follows from $\halflatticemod{-1}{\lambda}$ being top-generated, also using a similar argument to the $i=2$ case. Hence $(W_3)_{-m} u \otimes v_\text{top} \in N$ for all $m>0$. 

That $(W_i)_{-m} u \otimes v_\text{top} \in N$ can be established in the same way as in the $i=3$ case, where we use the expansions \eqref{eq:regcontentcasimirs} to reduce to the $j<i$ cases and use that $\halflatticemod{-1}{\lambda}$ is top-generated. Hence we conclude that $M \otimes {\halflatticemod{-1}{\lambda}}_\text{top} \subset N$ as required.
\end{proof}

\begin{corollary} \label{cor:almostirredimplies}
If $M$ is an almost-irreducible weight $\uRegkk$-module then $\tpmod{M}{-1,\lambda}$ is an almost-irreducible weight $\uSubkk$-module for all $\lambda \in \CC$.
\end{corollary}

\noindent We conclude this section with the case of $M$ being an irreducible $\uRegkk$-module. This is particularly important when we eventually want to discuss $\sRegkk$ and its modules: $\sRegkk$ is irreducible as a $\uRegkk$-module and, for nondegenerate admissible $\kk$, $\sRegkk$-modules are all direct sums of irreducible $\uRegkk$-modules. 

\begin{proposition} \label{prop:relaxedmodfromhw}
Let $M$ be an irreducible $\uRegkk$-module. Then 
\begin{itemize}
\item $\tpmod{M}{-1,\lambda}$ is an indecomposable relaxed highest-weight $\uSubkk$-module for all $\lambda \in \CC$.
\item $\tpmod{M}{-1,\lambda}$ is irreducible for almost all $\lambda \in \CC$.
\end{itemize}
\end{proposition}
\begin{proof}
As $M$ is irreducible, $M_\text{top}$ is a module for $\zhu{\uRegkk}$ which is abelian \cite{AraRat15}. Therefore, $M_\text{top}$ is one-dimensional and spanned by a vector $v_\gamma$ with $\gamma = (\gamma_2, \dots, \gamma_{n+1})$ defined by
\begin{equation}
(W_i)_0 \ v_\gamma = \gamma_i v_\gamma.
\end{equation}
The top space of $\tpmod{M}{-1,\lambda}$ is spanned by the set $\{ v_\gamma \otimes \ee^{-b + (\lambda +m)c} \vert m \in \ZZ \}$. To show that $\tpmod{M}{-1,\lambda}$ is a relaxed highest-weight $\uSubkk$-module, we need to find a relaxed highest-weight vector that generates $\tpmod{M}{-1,\lambda}$. As $\tpmod{M}{-1,\lambda}$ is top-generated by Theorem \ref{thm:topgenerated}, it suffices to find an $m$ for which $v_\gamma \otimes \ee^{-b + (\lambda +m)c}$ generates $\tpmod{M}{-1,\lambda}_\text{top}$. The modes $G^-_0$ and $G^+_0$ act on $v_\gamma \otimes \ee^{-b + (\lambda +m)c}$ according to
\begin{equation} \label{eq:zhugplusgminus}
    G^+_0\left(v_\gamma \otimes \ee^{-b + (\lambda +m)c} \right) 
    =v_\gamma \otimes \ee^{-b + (\lambda +m+1)c}, \quad G^-_0\left(v_\gamma \otimes \ee^{-b + (\lambda +m)c} \right) = p(\gamma, \lambda+m) v_\gamma \otimes \ee^{-b + (\lambda +m-1)c}
\end{equation}
where $p(\gamma,x)$ is polynomial  in $x$ of order at most $n+1$ by Theorem \ref{thm:regcontent}. Choose $m' \in \ZZ$ such that  $\lambda+m'$ is strictly less than the real parts of all roots of $p(\gamma,x)$. As $G^+_0$ and $G^-_0$ act injectively on $v_\gamma \otimes \ee^{-b + (\lambda +m')c}$, $v_\gamma \otimes \ee^{-b + (\lambda +m)c}$ is a relaxed highest-weight vector of $\tpmod{M}{-1,\lambda}$ that generates $\tpmod{M}{-1,\lambda}_\text{top}$ and therefore $\tpmod{M}{-1,\lambda}$.

That $\tpmod{M}{-1,\lambda}$ is indecomposable follows from $\tpmod{M}{-1,\lambda}$ being uniserial as in the $n=2$ case described in the proof of Theorem 5.12 of \cite{AdaRea20}. 

Finally, recall that $\tpmod{M}{-1,\lambda}$ has only top-submodules so any submodule of $N \subset \tpmod{M}{-1,\lambda}$ must contain an element of the form $\ee^{-b + (\lambda +m)c}$. As $G^+_0$ acts injectively on $\tpmod{M}{-1,\lambda}_\text{top}$, there exists $m' \le m$ such $G_0^- \ee^{-b + (\lambda +m')c} = 0$. That is, $\lambda+m'$ is a root of $p(\gamma,x)$. As  $p(\gamma,x)$ is polynomial in $x$ for fixed $\gamma$, there are finitely many $[\lambda] \in \CC / \ZZ$ such that $[\lambda]$ contains a root of $p(\gamma,x)$. 
\end{proof}

\begin{corollary}
Let $M$ be an irreducible $\uRegkk$-module. Conjugate highest-weight vectors in $\tpmod{M}{-1,\lambda}$ are of the form $v_\gamma \otimes \ee^{-b + (\lambda +m)c}$ with $ p(\gamma, \lambda+m) = 0$. If conjugate highest-weight vectors are present (i.e. when $\tpmod{M}{-1,\lambda}$ is reducible), let $m' \in \ZZ$ be the maximal $m$ satisfying $ p(\gamma, \lambda+m) = 0$. Then the submodule of $\tpmod{M}{-1,\lambda}$ generated by $v_\gamma \otimes \ee^{-b + (\lambda +m')c}$ is an irreducible conjugate highest-weight $\uSubkk$-module.
\end{corollary}

\noindent Given $\gamma \in \CC^{n}$, it is not immediately clear what the roots of the polynomial $p(\gamma,x)$ are. For $n=1$ and $2$, the roots of $p(\gamma,x)$ can be described using data from quantum hamiltonian reductions of certain highest-weight $\saffvoa{\kk}{\sltwo}$- and $\saffvoa{\kk}{\slthree}$-modules respectively. The $n=2$ case is essentially the second point in Theorem 4.20 from \cite{Feh21}.

It is expected that such a description holds for general $n$ but knowledge about the subregular quantum hamiltonian reduction functor is not currently sufficient to address this problem.

\section{Simple quotients} \label{sec:simple}

A natural question to ask is when the embedding $\uSubkk \hookrightarrow \uRegkk \otimes \Pi$ descends to an embedding of simple quotients. As we will see, this is almost-always true and depends on the level $\kk$. The restrictions for $\kk$ are known for $n=1$ \cite{AdaRea17} and $n=2$ \cite{AdaRea20}:
\begin{equation} \label{eq:simpleembn=12}
\begin{gathered}
    \saffvoa{\sltwo}{\kk}  \cong \sSub{\kk}{2} \hookrightarrow \Wsymb_{2,\kk} \otimes \Pi \cong \VOA{Vir}_\kk \otimes \Pi \Leftrightarrow \kk + 1 \notin \ZZ_{\ge 1} \\
    \VOA{BP}_\kk \cong \sSub{\kk}{3} \hookrightarrow \Wsymb_{3,\kk} \otimes \Pi \cong \VOA{Z}_\kk \otimes \Pi \Leftrightarrow \kk + 2, \ 2\kk+4  \notin \ZZ_{\ge 1}
\end{gathered}
\end{equation}
A nice feature of \eqref{eq:simpleembn=12} is that when $\kk$ is admissible, these conditions exclude precisely the degenerate admissible levels. Then, results similar to Proposition \ref{prop:relaxedmodfromhw} allow for the construction of continuous families of almost-always simple relaxed $\sSubkk$-modules for $n=1$ and $2$. 

In the $n=1$ and $2$ cases, the modules constructed in this manner are referred to as \emph{standard modules} and play a fundamental role in the determination of modular transformations and Grothendieck fusion rules for the simple subregular W-algebra at nondegenerate admissible levels. This is because the corresponding information for the simple regular W-algebra at these levels (also known as \emph{$\VOA{W}_{n+1}$ minimal models}) is known and the standard modules allows us to `lift' this information using the relaxed modules defined by inverse quantum hamiltonian reduction. 

Generalising this story to the $n>2$ case fully is out of the scope of this paper but is expected to follow the same lines. 

\subsection{Embedding for simple quotients}

Explicit formulae for singular vectors in $\uSubkk$ are only known for particular pairs of $n$ and $\kk$. When $n=1$, the Malikov--Feigin--Fuchs formula for singular vectors in Verma modules for $\affine{\sltwo}$ \cite{MFF} can be used to describe singular vectors in admissible-level $\saffvoa{\kk}{\sltwo}$ \cite{AdaVer95}. Singular vectors for $\uSub{\kk}{3}$ are known for $\kk=-\frac{5}{3}$ and $\kk=-\frac{9}{4}$ \cite{AdaCla19} in addition to admissible levels of the form $\kk = -3+\frac{\uu}{2}$ where $\uu>2$ is odd \cite{AraRat16}. 

Little is known about singular vectors and the corresponding $\uSubkk$-submodules of $\uSubkk$ for general $n$ and $\kk$. Fortunately, determining when embeddings of simple quotients exist only requires the knowledge we have about relaxed $\uSubkk$-modules from \cref{subsec:relaxedmod} and an understanding of singular vectors of a particular form:

\begin{proposition} \label{prop:singvect}
The vector $(G^+_{-1})^m \wun$, $m>0$, is singular in $\uSubkk$ if and only if $i(\kk+n) = m$ for some $i \in \{1, \dots, n\}$.
\end{proposition}
\begin{proof}
By charge and conformal dimension considerations (that is, that the eigenvalues of $L_0$ and $L_0-J_0$ are non-negative), it is clear that $J_s$, $L_s$, $(U_i)_s$, $G^+_{s-1}$ and $G^-_{s+1}$ annihilate $(G^+_{-1})^m \wun$ for all $s>0$. All that remains to check is when $G^-_1 (G^+_{-1})^m \wun = G^-_{(n)} (G^+)^m = 0$. By the embedding, we can identify $(G^+)^m$ with $\ee^{m c}$ and $G^-$ with $-\rho_{n+1}\dots\rho_1 \ee^{-c}$. Again using charge and conformal dimension considerations, $\left((G^+)^m\right)_{(n')}G^- = 0$  for all $n'>n$ and therefore
\begin{equation}
    G^-_{(n)} (G^+_{-1})^m = (-1)^{n+1} (\ee^{m c})_{(n)} \left(-\rho_{n+1}\dots\rho_1 \ee^{-c} \right).
\end{equation}
While this appears to have superficially made things more complicated, the above form allows us to explicitly compute the $n$-th product by performing the computation in $\uRegkk \otimes \Pi$. The following technical lemma is the most tedious part of this proof but the result gives the desired conditions on $\kk$ automatically. The idea is to not compute the $n$'th product directly but to sneak up on it by gradually inserting more $\rho_j$ operators while simultaneously raising the mode index on $\ee^{m c}$.

\begin{lemma}
Let $j \in \{0,1, \dots, n \}$. Then
\begin{equation} \label{eq:ithproduct}
    (\ee^{m c})_{(j)} \left(-\rho_{j+1}\dots\rho_1 \ee^{-c} \right) = m \prod_{i=1}^j \left( i(\kk+n)-m \right)\ee^{(m-1) c} ,
\end{equation}
where we take the product over $i$ to be equal to $1$ when $j=0$.
\end{lemma}
\begin{proof}
Proceeding by induction, the $j=0$ case can be checked directly.
\begin{equation}
\begin{aligned}
    (\ee^{m c})_{(0)} \left(-\rho_1 \ee^{-c} \right) &= -(\ee^{m c})_{(0)} \left( (\kk+n) (T_{-1}+t_{-1}) + b_{-1} + \frac{\kk+n+1}{n+1} c_{-1} - (\varepsilon_1)_{-1}\right) \ee^{-c} \\
    & = -(\ee^{m c})_{(0)} \left( a_{-1}  \ee^{-c} \right) \\
    & = \left[a_{-1},(\ee^{m c})_{(0)} \right]  \ee^{-c} \\
    & = m \ee^{(m-1) c}.
\end{aligned}
\end{equation}
For the inductive step, suppose that \eqref{eq:ithproduct} holds for some $j$. By the same charge and conformal dimension considerations used earlier, $(\ee^{m c})_{(j')} \left(-\rho_{j+1}\dots\rho_1 \ee^{-c} \right) = 0$ for all $j' > j$. To see that \eqref{eq:ithproduct} being true for $j$ implies that \eqref{eq:ithproduct} is true for $j+1$, we simply need to expand $\rho_{j+2}$ and reduce back to the $j$ case. That is,
\begin{equation} \label{eq:partialj+1prod}
\begin{aligned}
    (\ee^{m c})_{(j+1)} \left(-\rho_{j+2}\rho_{j+1}\dots\rho_1 \ee^{-c} \right) &= (\ee^{m c})_{(j+1)} \left(-\left((\kk+n)\partial + b_{-1}  + \frac{\kk+n+1}{n+1}c_{-1} - (\varepsilon_{j+2})_{-1} \right)\rho_{j+1}\dots\rho_1 \ee^{-c} \right) \\
    &= (\ee^{m c})_{(j+1)} \left(-\left((\kk+n)t_{-1} + b_{-1}  + \frac{\kk+n+1}{n+1}c_{-1} \right)\rho_{j+1}\dots\rho_1 \ee^{-c} \right) \\
    &= -(\kk+n)(\ee^{m c})_{(j+1)} \left(t_{-1}\rho_{j+1}\dots\rho_1 \ee^{-c} \right) - m (\ee^{m c})_{(j)} \left(-\rho_{j+1}\dots\rho_1 \ee^{-c} \right)
\end{aligned}
\end{equation}
using the commutation relation, for $p, q \in \ZZ$,
\begin{equation}
    \left[ A c_{(p)} + B d_{(p)}, \ee^{m c}_{(q)} \right] = 2 B m \ \ee^{m c}_{(p+q)}.
\end{equation}
The term involving $t_{-1}$ can be simplified using standard identities involving derivatives and $i$'th products described in, for example, \cite{KacVer94}. Here we attack this head-on using the operator product expansions of $\ee^{mc}(z)$ and  $\left(-\rho_{j+1}\dots\rho_1 \ee^{-c} \right)(z)$. This is of course equivalent to working with $i$'th products. By the inductive hypothesis and the observation earlier that the $j'$'th product vanishes for $j' > j$,
\begin{equation}
    \ee^{mc}(z)\left(-\rho_{j+1}\dots\rho_1 \ee^{-c} \right)(w) \sim \frac{m \prod_{i=1}^j \left( i(\kk+n)-m \right)\ee^{(m-1) c}(w)}{(z-w)^{j+1}} + \dots .
\end{equation}
Applying $\partial_w$ to both sides and extracting the coefficient field of the order $j+2$ pole finally gives 
\begin{equation}
    -(\kk+n)(\ee^{m c})_{(j+1)} \left(t_{-1}\rho_{j+1}\dots\rho_1 \ee^{-c} \right) = (j+1)(\kk+n) m \prod_{i=1}^j \left( i(\kk+n)-m \right)\ee^{(m-1) c} .
\end{equation}
Combining this with \eqref{eq:partialj+1prod} shows that \eqref{eq:ithproduct} being true for $j$ implies that \eqref{eq:ithproduct} is also true for $j+1$ and therefore, by induction, we have our desired result.
\end{proof}

\noindent Substituting $j=n$ into \eqref{eq:ithproduct} gives
\begin{equation}
    G^-_{(n)} (G^+_{-1})^m = (-1)^{n+1} m \prod_{i=1}^n \left( i(\kk+n)-m \right)(G^+_{-1})^{m-1},
\end{equation}
and from here it is clear that the right-hand-side vanishes exactly when $i(\kk+n) = m$ for some $i \in \{1 ,\dots, n \}$ as required.
\end{proof}

\noindent Substituting $n=1$ and $2$ reproduces the conditions for $\saffvoa{\kk}{\sltwo}$ and $\VOA{BP}_\kk$ respectively. With this result in hand, we are in a position to answer the question of when our embedding $\uSubkk \hookrightarrow \uRegkk \otimes \Pi$ descends to an embedding of simple quotients. Let $\psi_\kk$ denote the composition of the embedding $\uSubkk \hookrightarrow \uRegkk \otimes \Pi$ with the projection from $\uRegkk$ to its simple quotient $\sRegkk$. That is, 
\begin{equation}
    \psi_\kk: \uSubkk \hookrightarrow \uRegkk \otimes \Pi \twoheadrightarrow \sRegkk \otimes \Pi.
\end{equation}
This is clearly non-zero as image of the vacuum of $\uSubkk$ is the vacuum of $\sRegkk \otimes \Pi$. In terms of $\psi_\kk$, the question at hand then becomes: for which $\kk$ is $\psi_\kk(\uSubkk)$ simple, i.e. $\psi_\kk(\uSubkk) \cong \sSubkk$?  

\begin{theorem} \label{thm:descent}
$\sSubkk$ embeds into $\sRegkk \otimes \Pi$ if and only if $i(\kk+n) \notin \ZZ_{\ge 1}$ for all $i \in \{1, \dots, n\}$ .
\end{theorem}
\begin{proof}
The proof of this statement is very similar to the $n=2$ case presented in Theorem 6.2 of \cite{AdaRea20}. Indeed the only modification in our present context is the slightly different definition of the relevant $\Pi$-modules and having to care about the action of modes from a few more fields. 

Following the proof provided therein, suppose $\psi_\kk(\uSubkk)$ is not simple as a $\uSubkk$-module and therefore has a non-zero proper ideal $I$. As $\uSubkk$-modules, $I \subset \sRegkk \otimes \Pi = \sRegkk \otimes \halflatticemod{0}{0}$. Applying the spectral flow $\uSubkk$-automorphism $\sfsymb^{-1}$ to both sides and observing that spectral flow can be realised using $\uRegkk$- and $\Pi$-automorphisms according to $\sfsymb^{-1} = \text{id}_r \otimes \rho^{-1}$,
\begin{equation}
    \sfmod{-1}{I} \subset \sRegkk \otimes \halflatticemod{-1}{0} = \tpmod{\sRegkk}{-1,0}.
\end{equation}
As $\sRegkk$ is an irreducible weight $\uRegkk$-module, $\tpmod{\sRegkk}{-1,0}$ is an almost-irreducible $\uSubkk$-module by Corollary \ref{cor:almostirredimplies}. The top space of $\tpmod{\sRegkk}{-1,0}$ is spanned by the vectors $\wun_R \otimes \ee^{-b+ m c}$. Hence there exists $m \in \ZZ$ such that $\wun_R \otimes \ee^{-b+ m c} \in \sfmod{-1}{I}$ and therefore after applying spectral flow again $\wun_R \otimes \ee^{ m c} \in I$. 

As $\uSubkk$ contains only fields of nonnegative conformal dimension, it must be that $m \ge 0$. Even better, if $m=0$ then $I$ contains the vacuum of $\uRegkk$ and therefore $I = \psi_\kk(\uSubkk)$ which is a contradiction. Take $m>0$ to be minimal satisfying $\wun_R \otimes \ee^{m c} \in I$. In particular, $\wun_R \otimes \ee^{(m-1) c} \notin I$. 

As $\psi_\kk((G^+)^m) = \wun_R \otimes \ee^{m c}$ is annihilated by all $J_s$, $L_s$, $(U_i)_s$, $G^+_{s-1}$ and $G^-_{s}$ with $s>0$, it is a singular vector in $I$. As the embedding of Theorem \ref{thm:invemb} sends $(G^+)^{\ell}$ to $\wun_R \otimes \ee^{\ell c}$ for all $\ell \ge 0$ and composing with the projection map $\uRegkk \twoheadrightarrow \sRegkk$ leaves $\Pi$ untouched, $\psi_\kk((G^+)^{m-1}) = \wun_R \otimes \ee^{(m-1) c}$ is non-zero. Hence $(G^+)^{m}$ is singular in $\uSubkk$. By Proposition \ref{prop:singvect}, if $i(\kk+n) \neq m$ for some $i \in \{ 1, \dots, n\}$ then this cannot occur and therefore $\psi_\kk(\uSubkk)$ is simple. 

For the converse, if $i(\kk+n) = m$ for some $i \in \{ 1, \dots, n\}$ then $(G^+)^{m}$ is singular in $\uSubkk$. Then $\psi_\kk((G^+)^{m}) \neq 0$ is singular in $\psi_\kk(\uSubkk)$ and therefore $\psi_\kk(\uSubkk)$ is not simple. 
\end{proof}

\subsection{Admissible levels and standard modules} \label{sec:standard}

\noindent Anticipating that admissible levels are both interesting and important from a logarithmic conformal field theory point of view, it is useful to know when admissible-level $\sSubkk$ is related to admissible-level $\sRegkk$ using the inverse quantum hamiltonian reduction of Theorem \ref{thm:descent}. Recall that an admissible level $\kk$ for $\slnpone$ is one that satisfies 
\begin{equation}
    \kk+n+1 = \frac{\uu}{\vv}, \text{ where } \uu \in \ZZ_{\ge n+1}, \vv \in \ZZ_{\ge 1} \text{ and } \text{gcd}\{\uu,\vv \} = 1.
\end{equation}
\begin{corollary} \label{cor:minmodemb}
Let $\kk = -n-1 + \fracuv$ be admissible. $\sSubkk$ embeds into $\sRegkk \otimes \Pi$ if and only if $\vv > n$. That is, 
\begin{equation}
    \Subminmod{\uu}{\vv} \hookrightarrow \Regminmod{\uu}{\vv} \otimes \Pi.
\end{equation}
\end{corollary}
\begin{proof}
Suppose that $\kk = -n-1 + \fracuv$ with $\vv \le n$. Then $v(\kk+n) = \uu - \vv \in \ZZ_{\ge 1}$ as $\uu \in \ZZ_{\ge n+1}$. So for such $\kk$, $\sSubkk$  does not embed into $\sRegkk \otimes \Pi$. If $\vv > n$ (i.e. $\kk$ is nondegenerate admissible), $i(\kk+n) = i \fracuv - i$ is not an integer for all  $i \in \set{1, \dots, n}$ so $\Subminmod{\uu}{\vv} \hookrightarrow \Regminmod{\uu}{\vv} \otimes \Pi$ by Theorem \ref{thm:descent}.
\end{proof}

\noindent Theorem \ref{thm:descent} also shows there exists an embedding $\sSubkk \hookrightarrow \sRegkk \otimes \Pi$ for certain non-admissible levels. This includes fractional levels of the form $\kk = -n-1+\frac{n}{n+1}$, at which $\uRegkk$ is reducible \cite{Ara07}.

Mirroring the construction of relaxed modules for the universal subregular W-algebra $\uSubkk$ using the inverse reduction embedding, relaxed modules for the simple subregular W-algebra $\sSubkk$ can be constructed when the embedding $\sSubkk \hookrightarrow \sRegkk \otimes \Pi$ exists.

At non-admissible level, the representation theory of the affine vertex operator algebras and W-algebras is largely mysterious. Despite this, Proposition \ref{prop:relaxedmodfromhw} shows that $\sSubkk$ admits infinitely many irreducible modules of the form $\tpmod{\sRegkk}{-1,\lambda}$ for $\kk = -n-1+\frac{n}{n+1}$, and likely for many more non-admissible levels. 

Returning to a nondegenerate admissible level $\kk = -n-1+\fracuv$, $\sRegkk = \Regminmod{\uu}{\vv}$ is rational and hence has finitely many irreducible modules denoted by $\mathcal{W}_\gamma$ for $\gamma \in \text{Pr}^k_\mathcal{W}$ as described in Section \ref{sec:reg}. Therefore by Corollary \ref{cor:minmodemb} and Proposition \ref{prop:relaxedmodfromhw}, the $\Subminmod{\uu}{\vv}$-module $\tpmod{\mathcal{W}_\gamma}{-1,\lambda}$ is an indecomposable relaxed highest-weight module that is almost-always irreducible. This shows that $\Subminmod{\uu}{\vv}$ is nonrational in the category of weight modules.

The relaxed modules $\tpmod{\mathcal{W}_\gamma}{-1,\lambda}$ play the role of `standard modules' in the computation of modular transformations and fusion rules of $\Subminmoduv$ when $n=1$ \cite{CreMod13} and $n=2$ \cite{FehilyMod}. Much of the representation-theoretic structure present in those cases is present for all $n$. It is therefore expected that the relaxed $\Subminmoduv$-modules $\tpmod{\mathcal{W}_\gamma}{-1,\lambda}$ will be essential in computing logarithmic conformal field theoretic data when $n>2$ as well. 

The simple subregular W-algebra $\Subminmod{n+1}{n+2}$ is isomorphic to the logarithmic $\mathcal{B}_{n+1}$-algebra \cite{CreCos13, Aug20, AdaRPVP21}. Tensor categories related to $\mathcal{B}_{n+1}$ that are  braided, rigid and non-semisimple have been constructed using a conjectural relationship between $\mathcal{B}_{n+1}$ and the unrolled restricted quantum groups of $\sltwo$ \cite{Aug20}. It would be interesting to see if our representation-theoretic results for $\Subminmod{n+1}{n+2}$ are able to reproduce this categorical data.

Something lacking in the general $\Subminmod{\uu}{\vv}$ case is an understanding of highest-weight (and conjugate highest-weight) $\Subminmoduv$-modules for $\vv > n$. Much of this is due to the complexity of the operator product expansion between $G^+(z)$ and $G^-(z)$, and the related problem of finding roots of the polynomial $p(\gamma,x)$ from \eqref{eq:zhugplusgminus}.

More detailed knowledge about how the subregular quantum hamiltonian reduction functor acts on highest-weight $\saffvoa{\kk}{\slnpone}$-modules will likely assist in this direction. The structure of highest-weight $\Subminmoduv$-modules (in particular the vacuum module) is also required to compute fusion rules of $\Subminmoduv$ using inverse quantum hamiltonian reduction as has been done for $n=1$ \cite{CreMod13} and $2$ \cite{FehilyMod}.

\newpage
\appendix
\section{Operator product expansions of subregular W-algebras} \label{app:opes}
Obtaining operator product expansions for $\uSubkk$ becomes exceedingly difficult as $n$ gets large. For fixed $n$, one can in principle use the construction of W-algebras from Kac and  Wakimoto \cite{KacQua04} to describe strong generators for $\uSubkk$ and their operator product expansions. 

In light of the results of Section \ref{subsec:explicit}, operator product expansions can equivalently be obtained from the free-field realisation in terms of fields in $\uRegkk$ and $\Pi$. The operator product expansions listed in this appendix are those satisfied by the strong generators in Section \ref{subsec:explicit} for $n=1$, $2$ and $3$.  

\subsection{Type \texorpdfstring{$A_1$}{A1}} \label{subsec:A1} 
$\uSub{\kk}{2}$ is the universal affine vertex algebra $\uaffvoa{\kk}{\sltwo}$. It is the vertex algebra with vacuum $\wun$ strongly and freely generated by fields $h(z) = 2 J(z)$, $e(z) = G^+(z)$ and $f(z) = G^-(z)$ satisfying the following \opes{}:
\begin{equation} \label{ope:sl2}
\begin{gathered}
    h(z)h(w) \sim  \frac{2\kk \wun}{(z-w)^2}, \qquad
    h(z)e(w) \sim  \frac{2e(w)}{(z-w)}, \qquad
    h(z)f(w) \sim  \frac{-2f(w)}{(z-w)}, \\
    e(z)e(w) \sim  f(z)f(w) \sim 0, \qquad
	e(z)f(w) \sim \frac{\kk \wun}{(z-w)^2} + \frac{h(w)}{z-w}.
\end{gathered}
\end{equation}
$\uSub{\kk}{2}$ can be given a conformal structure (and therefore becomes a vertex operator algebra) using the \emph{Sugawara construction} as long as $\kk \neq -2$. The energy-momentum field is given by
\begin{equation}
    L(z) = \frac{1}{2(\kk+2)}\left( \frac{1}{2}\no{h(z)h(z)} + \no{e(z)f(z)} + \no{f(z)e(z)}   \right).
\end{equation}
 With respect to $L(z)$, the conformal dimension of the generating fields $h(z)$, $e(z)$ and $f(z)$ is 1, whilst the central charge of $\uSub{\kk}{2}$ is
\begin{equation}
	\overline{\cc}^{2}_{\kk} = \frac{3\kk}{\kk+2}.
\end{equation}

\subsection{Type \texorpdfstring{$A_2$}{A2}} \label{subsec:A2} 

$\uSub{\kk}{3}$ is isomorphic to the Bershadsky--Polyakov algebra $\ubpvoak$. It is the vertex operator algebra with vacuum $\wun$ that is strongly and freely generated by fields $J(z)$, $G^+(z)$, $G^-(z)$ and $L(z)$ satisfying the following \opes{}:
\begin{equation} \label{ope:bp}
\begin{gathered}
    L(z)L(w) \sim -\frac{2(\kk+1)(2\kk+3) \wun}{ (\kk+3) (z-w)^4} + \frac{2L(w)}{(z-w)^2} + \frac{\partial L(w)}{(z-w)}, \qquad
    L(z)J(w) \sim -\frac{(2\kk+3) \wun}{3(z-w)^3} + \frac{J(w)}{(z-w)^2} + \frac{\partial J(w)}{(z-w)}, \\
    L(z)G^+(w) \sim \frac{G^+(w)}{(z-w)^2} + \frac{\partial G^+(w)}{(z-w)}, \qquad
    L(z)G^-(w) \sim \frac{2G^-(w)}{(z-w)^2} + \frac{\partial G^-(w)}{(z-w)},  \\
    J(z)J(w) \sim  \frac{(2\kk+3) \wun}{3(z-w)^2}, \qquad
    J(z)G^\pm(w) \sim  \frac{\pm G^\pm (w)}{(z-w)}, \qquad
    G^\pm(z)G^\pm(w) \sim 0, \\
	G^+(z)G^-(w) \sim \frac{(\kk+1)(2\kk+3) \wun}{(z-w)^3} + \frac{3(\kk+1) J(w)}{(z-w)^2} + \frac{3 \no{J(w)J(w)} + (2\kk+3) \partial J(w) - (\kk+3) L(w)}{z-w}.
\end{gathered}
\end{equation}
This family of \voas\ was first described in \cite{PolGau90,BerCon91} where it was constructed via a `new' type of \qhr\ from the universal affine \voas\ $\uslvoak$ associated to $\slthree$. From \eqref{ope:bp}, we see that the conformal dimensions of the generating fields $J(z)$, $G^+(z)$, $G^-(z)$ and $L(z)$ are $1$, $1$, $2$ and $2$ respectively, whilst the central charge of $\uSub{\kk}{3}$ is
\begin{equation}
	\overline{\cc}^{3}_{\kk} = -\frac{4(\kk+1)(2\kk+3)}{\kk+3}.
\end{equation}
\newpage
\subsection{Type \texorpdfstring{$A_3$}{A3}} \label{subsec:A3} 
$\uSub{\kk}{4}$ is isomorphic to the $n=4$ Feigin--Semikhatov algebra $\FS{4}$. It is the vertex operator algebra with vacuum $\wun$ that is strongly and freely generated by fields $J(z)$, $G^+(z)$, $G^-(z)$, $W(z)$ and $L(z)$ satisfying the following \opes{}:
{\small\begin{equation} \label{ope:FS}
\begin{gathered}
    L(z)L(w) \sim -\frac{(3\kk+8)(11\kk+29) \wun}{2 (\kk+4) (z-w)^4} + \frac{2L(w)}{(z-w)^2} + \frac{\partial L(w)}{(z-w)}, \qquad
    L(z)J(w) \sim -\frac{(3\kk+8) \wun}{2(z-w)^3} +\frac{J(w)}{(z-w)^2} + \frac{\partial J(w)}{(z-w)}, \\
	L(z)W(w) \sim \frac{3W(w)}{(z-w)^2} + \frac{\partial W(w)}{(z-w)}, \qquad	    
    L(z)G^+(w) \sim \frac{G^+(w)}{(z-w)^2} + \frac{\partial G^+(w)}{(z-w)}, \qquad
    L(z)G^-(w) \sim \frac{3G^-(w)}{(z-w)^2} + \frac{\partial G^-(w)}{(z-w)},  \\
    J(z)J(w) \sim  \frac{(3\kk+8) \wun}{4(z-w)^2}, \qquad
    J(z)G^\pm(w) \sim  \frac{\pm G^\pm (w)}{(z-w)}, \qquad
    J(z)W(w) \sim 0, \qquad
    G^\pm(z)G^\pm(w) \sim 0,\\
	G^+(z)G^-(w) \sim \frac{(\kk+2)(2\kk+5)(3\kk+8) \wun}{(z-w)^4} + \frac{4(\kk+2)(2\kk+5) J(w)}{(z-w)^3} -\frac{(\kk+2)\big((\kk+4)\tilde{L}(w) - 6 \no{J(w) J(w)} - 2(2\kk+5)\partial J(w)\big)}{(z-w)^2} \\
	\hspace{-50pt}+(\kk+2)\Bigg( W(w)+\frac{8(11\kk+32)}{3(3\kk+8)^2} \no{J(w)^3}  -\frac{4(\kk+4)}{3\kk+8}\no{\tilde{L}(w) J(w)} + 6\no{J(w)	\partial J(w)}\\
	\hspace{10pt}-\frac{1}{2}(\kk+4)\partial \tilde{L}(w) + \frac{4(3\kk^2+17\kk+26)}{3(3\kk+8)}\partial^2 J(w) \Bigg)(z-w)^{-1}, \\
	W(z)G^\pm(w) \sim \pm \frac{2(\kk+4)(3\kk+7)(5\kk+16)G^\pm(w)}{(3\kk+8)^2 (z-w)^3}+ \left(\pm \frac{3(\kk+4)(5\kk+16)}{2(3\kk+8)}\partial G^\pm (w) - \frac{6(\kk+4)(5\kk+16)}{(3\kk+8)^2}\no{J(w) G^\pm(w)}\right) (z-w)^{-2}\\
	\hspace{50pt}+\Bigg(-\frac{8(\kk+3)(\kk+4)}{(\kk+2)(3\kk+8)}\no{J(w) \partial G^\pm(w)}-\frac{4(\kk+4)(3\kk^2+15\kk+16)}{(\kk+2)(3\kk+8)^2} \no{\partial J(w)G^\pm(w)}  \pm \frac{(\kk+3)(\kk+4)}{\kk+2}\partial^2 G^\pm(w)\\
	\hspace{30pt}\mp \frac{2(\kk+4)^2}{(\kk+2)(3\kk+8)} \no{\tilde{L}(w) G^\pm(w)} \pm \frac{4(\kk+4)(5\kk+16)}{(\kk+2)(3\kk+8)^2}\no{J(w)^2 G^\pm(w)}\Bigg)(z-w)^{-1},\\
	W(z)W(w) \sim \frac{2(\kk+4)(2\kk+5)(3\kk+7)(5\kk+16) \wun}{(3\kk+8)(z-w)^6}-\frac{3(\kk+4)^2(5\kk+16) \tilde{L}_\perp(w)}{(3\kk+8)(z-w)^4}-				\frac{3(\kk+4)^2(5\kk+16)\partial \tilde{L}_\perp(w)}{2(3\kk+8)(z-w)^3}\\
	\hspace{-100pt}+ \Bigg( -\frac{3(\kk+4)^2(5\kk+16)(12\kk^2+59\kk+74)}{4(3\kk+8)(20\kk^2+93\kk+102)}\partial^2 \tilde{L}_\perp(w) \\
	\hspace{100pt}+\frac{8(\kk+4)^3(5\kk+16)}{(3\kk+8)(20\kk^2 + 93\kk+102)}\no{ \tilde{L}_\perp(w) \tilde{L}_\perp(w)} +4(\kk+4)\Lambda(w) \Bigg) (z-w)^{-2}\\
	\hspace{-70pt}+\Bigg(-\frac{(\kk+4)^2(5\kk+16)(12\kk^2+59\kk+74)}{6(3\kk+8)(20\kk^2+93\kk+102)}\partial^3 \tilde{L}_\perp(w) \\
	\hspace{100pt} +\frac{8(\kk+4)^3(5\kk+16)}{(3\kk+8)(20\kk^2+93\kk+102)}\no{\partial \tilde{L}_\perp(w)\tilde{L}_\perp(w)} + 2(\kk+4)\partial \Lambda(w) \Bigg)(z-w)^{-1},
\end{gathered}
\end{equation} }%
where $\tilde{L}(z) = L(z) - \partial J(z)$, $\tilde{L}_\perp(z) = \tilde{L}(z) - \frac{2}{3\kk+8}\no{J(z)J(z)}$ and
{\small\begin{equation}
\begin{aligned}
    (\kk+2)^2 \Lambda(z) =& \ \no{G^+(z) G^-(z)} - \frac{\kk+2}{2}\partial W(z) -\frac{4(\kk+2)}{3\kk+8}\no{W(z) J(z)}+\frac{3(\kk+2)^2(\kk+4)(6\kk^2+33\kk+46)}{2(3\kk+8)(20\kk^2+93\kk+102)}\partial^2 \tilde{L}_\perp(z) \\
	&\hspace{-35pt}-\frac{(\kk+2)(\kk+4)^2(11\kk+26)}{2(3\kk+8)(20\kk^2+93\kk+102)}\no{\tilde{L}_\perp(z) \tilde{L}_\perp(z)} + \frac{2(\kk+2)(\kk+4)}{3\kk+8}\partial \no{\tilde{L}_\perp(z) J(z)} +\frac{8(\kk+2)(\kk+4)}{(3\kk+8)^2} \no{\tilde{L}_\perp(Z) J(z) J(z)}\\
	&\hspace{-33pt}-\frac{(\kk+2)(2\kk+5)}{3\kk+8}\Bigg(\frac{8}{3}\no{\partial^2 J(z)J(z)} +2 \no{\partial J(z)\partial J(z)} +\frac{16}{3\kk+8}\no{\partial J(z)J(z) J(z)} + \frac{32}{3(3\kk+8)^2}\no{J(z)^4} + \frac{3\kk+8}{6}\partial^3 J(z) \Bigg).
\end{aligned}
\end{equation}}%
From \eqref{ope:FS}, we see that the conformal dimensions of the generating fields $J(z)$, $G^+(z)$, $W(z)$ $G^-(z)$ and $L(z)$ are $1$, $1$, $3$, $3$ and $2$ respectively, whilst the central charge of $\uSub{\kk}{4}$ is
\begin{equation}
	\overline{\cc}^{4}_{\kk} = -\frac{(3\kk+8)(11\kk+29) }{ \kk+4}.
\end{equation}

\newpage

\raggedright

\providecommand{\opp}[2]{\textsf{arXiv:\mbox{#2}/#1}}\providecommand{\pp}[2]{\textsf{arXiv:#1
  [\mbox{#2}]}}

\end{document}